\documentclass[reqno]{amsart}

\usepackage[english]{babel}
\usepackage{amsmath,amssymb,amsthm}
\usepackage{hyperref}
\usepackage[centering]{geometry}
\usepackage{xcolor}

\renewcommand{\epsilon}{\varepsilon}            

\renewcommand{\H}{\ensuremath{\mathbb{H}}}

\newcommand{\RR}{\mathbb{R}}
\newcommand{\CC}{\mathbb{C}}

\newcommand{\IIm}{\operatorname{Im}}

\DeclareMathOperator{\diag}{diag}

\newcommand{\br}[1]{\left(#1\right)}
\newcommand{\lang}[1]{\langle #1 \rangle}
\usepackage{nicematrix,tikz}
\newcommand{\trh}{\frac{1}{\sqrt{2}}}
\usepackage{float}

\usepackage{slashed}

\newtheorem{theorem}{Theorem}[section]   
\newtheorem*{theorem*}{Theorem}          
\newtheorem{lemma}[theorem]{Lemma}
\newtheorem{proposition}[theorem]{Proposition}

\theoremstyle{definition}

\newtheorem{corollary}[theorem]{Corollary}
\newtheorem{example}[theorem]{Example}

\newtheorem{remark}{Remark}[section]

\numberwithin{equation}{section}

\title[Complete classification of isoparametric hypersurfaces in $\mathbb{H}^m\times \mathbb{H}^n$]
{Complete classification of isoparametric hypersurfaces in $\mathbb{H}^m\times \mathbb{H}^n$}
\thanks{}

\author{Haizhong Li}
\author{Renhao Tan}
\author{Zeke Yao}


\date{}

\begin{document}
	
\begin{abstract}
In this paper, we classify isoparametric hypersurfaces in $\mathbb{H}^m\times \mathbb{H}^n$ ($m\geq3, n\geq2$). 
The key rigidity result is that every such hypersurface has constant principal curvatures. Combining this rigidity with the classification of hypersurfaces with constant principal curvatures and constant product angle function yields the complete classification.  Consequently, every isoparametric hypersurface in $\mathbb{H}^m\times \mathbb{H}^n$ 
is an open part of a homogeneous example. 
\end{abstract}
	
\maketitle
	
\section{Introduction}\label{sect:1}
	
Let $M$ be an orientable hypersurface of a Riemannian manifold $(\bar{M}, \bar g)$.
$M$ is called an isoparametric hypersurface of $\bar{M}$ if there exists an isoparametric function $F:\bar{M} \rightarrow \mathbb{R}$ such that $M=F^{-1}(r)$,
for some regular value $r$ of $F$. Here $F$ is called an isoparametric function
if the gradient and the Laplacian of $F$ satisfy
$$
\|\nabla F\|^{2}=f_1(F), \quad \Delta F=f_2(F),
$$
where $f_1, f_2: \mathbb{R} \rightarrow \mathbb{R}$ are smooth functions.
In addition to the above definition,
there is another equivalent characterization for isoparametric hypersurfaces. A hypersurface of a Riemannian manifold is isoparametric if and only if its locally defined parallel hypersurfaces
have constant mean curvature. 
	
For hypersurfaces in real space forms, the condition of being isoparametric is equivalent to having constant principal curvatures. The full classification of isoparametric hypersurfaces in Euclidean spaces was obtained by Segre \cite{Seg}, and in real hyperbolic spaces by Cartan \cite{Car}. The situation is more involved in spheres. Cartan classified hypersurfaces with $g \in \{1,2,3\}$ constant principal curvatures. Subsequently, Hsiang and Lawson \cite{HL2} gave a complete classification of homogeneous hypersurfaces in spheres, while Takagi-Takahashi \cite{TT} computed their principal curvatures. 
Later, M\"{u}nzner \cite{Muz} proved that for any isoparametric hypersurface in a sphere, the number of distinct principal curvatures necessarily satisfies $g \in \{1,2,3,4,6\}$. 
Isoparametric hypersurfaces in spheres have been studied extensively. We refer to the excellent surveys in \cite{CR15,Tho} and the references therein. 
	
When the ambient space has nonconstant sectional curvature, the isoparametricity of
a hypersurface is generally not equivalent to the constancy of the principal curvatures. 
The first counterexamples were constructed by Wang \cite{Wang1} in complex projective spaces. In complex projective space, Ge, Tang and Yan \cite{GTY} proved that
isoparametric hypersurfaces in $\mathbb{C}P^{2n}$ are homogeneous,  but this is no longer valid for isoparametric hypersurfaces in $\mathbb{C}P^{2n+1}$. Dom\'{\i}nguez-V\'{a}zquez \cite{D} and, subsequently, Dom\'{\i}nguez-V\'{a}zquez and Kollross \cite{DK} classified isoparametric hypersurfaces in complex projective spaces.  D\'{\i}az-Ramos, Dom\'{\i}nguez-V\'{a}zquez and Sanmart\'{\i}n-L\'{o}pez \cite{DDS} classified isoparametric hypersurfaces in complex hyperbolic spaces, among which there are isoparametric hypersurfaces with nonconstant principal curvatures. For real hypersurfaces in non-flat complex space form with constant principal curvatures, Kimura \cite{Kimura86} classified the Hopf hypersurfaces with constant principal curvatures, in particular, the number
of distinct principal curvatures is $g\in\{2,3,5\}$. In the complex hyperbolic space $\mathbb{C}H^m$ ($m\geq2$), Berndt \cite{B} classified the
Hopf hypersurfaces with constant principal curvatures, and in this case the number of
distinct principal curvatures is $g\in\{2,3\}$. 
Furthermore, in other canonical irreducible Riemannian symmetric spaces and $3$-dimensional homogeneous Riemannian manifolds, there have been some interesting results on construction and classification of isoparametric hypersurfaces (cf. \cite{DDR,D-M,DS24} etc.), on hypersurfaces with constant principal curvatures (cf. \cite{Berndt91,D-M,LTY25,LTY26,Ta2} etc.). 
	
Now we focus on the study of hypersurfaces in the product spaces. 
Urbano \cite{Ur} classified the isoparametric hypersurfaces and the hypersurfaces with at most two distinct constant principal curvatures in $\mathbb{S}^2\times\mathbb{S}^2$.  Later, Gao, Ma and Yao \cite{GMY24} classified the isoparametric hypersurfaces and the hypersurfaces with at most two distinct constant principal curvatures in $\mathbb{H}^2\times\mathbb{H}^2$. After that, Gao, Ma and Yao \cite{GMY24-2} classified the isoparametric hypersurfaces in other product spaces of two  $2$-dimensional space forms. 
Recently, de Lima and Pipoli \cite{dP24} classified the isoparametric hypersurfaces 
in $\mathbb{H}^n\times \mathbb{R}$ and $\mathbb{S}^n\times \mathbb{R}$. 
Tan, Xie and Yan \cite{TXY25} classified the isoparametric hypersurfaces in $\mathbb{H}^n\times \mathbb{R}^m$ and $\mathbb{S}^n\times \mathbb{R}^m$. Recently, de Lima and Pipoli \cite{DP} proved that every connected isoparametric hypersurface in a product of simply connected space forms has constant product angle function, and classified the isoparametric hypersurfaces in $\mathbb{S}^m\times \mathbb{H}^n$. 
Under additional hypotheses, they also obtained classification results for isoparametric hypersurfaces in $\mathbb{S}^m\times\mathbb{S}^n$ and $\mathbb{H}^m\times\mathbb{H}^n$. For further results on isoparametric hypersurfaces, the hypersurfaces with constant principal curvatures and other canonical hypersurfaces in the product spaces, we refer to \cite{Toj26,CS19,LVWY1,LVWY2,MT11,M25,Toj13} and the references therein. 
	
In this paper, we study the isoparametric hypersurfaces of $\mathbb{H}^m\times \mathbb{H}^n$ ($m\geq3, n\geq2$). Before stating our main results, we first recall that there is a natural  product structure $P$ on $\mathbb{H}^m\times\mathbb{H}^n$ defined by $P(v_1,v_2):=(v_1,-v_2)$ for any tangent vector fields $v_1\in T\mathbb{H}^m$ and $v_2\in T\mathbb{H}^n$. Then, for an orientable hypersurface $M$ of $\mathbb{H}^m\times\mathbb{H}^n$ with a unit normal vector field $N$,  we can introduce an important function $C$ defined by $C:=\langle PN, N\rangle$. Here, $\langle\cdot,\cdot\rangle$ denotes the standard product metric on $\mathbb{H}^m\times\mathbb{H}^n$. Let $V:=PN-CN$ be the tangential part of $PN$. The geometry of hypersurfaces in $\mathbb{H}^m\times\mathbb{H}^n$ is closely related to this function $C$. Hereafter, for the sake of brevity, we shall call $C$ the {\it product angle function} of $M$. Now, as our main result, we obtain the following classification result. 
	
\begin{theorem}\label{thm:isoparametric-HH}
Let $M$ be a connected oriented isoparametric hypersurface of $\mathbb{H}^m\times \mathbb{H}^n$ ($m\geq3, n\geq2$). Then up to isometries of $\mathbb{H}^m\times \mathbb{H}^n$, one of the following five cases occurs: 
\begin{itemize}
\item[(1)] $M$ is an open part of $\Sigma\times \mathbb{H}^n$, where $\Sigma$ is a hypersurface of $\mathbb{H}^m$ with constant principal curvatures (see Example \ref{exam:1}); or
\item[(2)] $M$ is an open part of $\mathbb{H}^m\times \tilde{\Sigma}$, where $\tilde{\Sigma}$ is a hypersurface of $\mathbb{H}^n$ with constant principal curvatures (see Example \ref{exam:2}); or
\item[(3)] $M$ is an open part of $M_{1,-1}^{C}$ for some $C\in (-1,1)$ (see Example \ref{exam:3}); or
\item[(4)] $M$ is an open part of $M_{1,1}^{C}$ for some $C\in (-1,1)$ (see Example \ref{exam:4}); or
\item[(5)] $m=n$, and $M$ is an open part of $M_\tau$ for some $\tau<-1$ (see Example \ref{exam:5}). 
\end{itemize}  
\end{theorem}
	
\begin{remark}
A central difficulty is that, in an ambient manifold of nonconstant sectional curvature, isoparametricity does not generally force the principal curvatures to be constant. Hence the main rigidity step in our argument is to show that this implication does hold in $\mathbb{H}^m\times \mathbb{H}^n$ ($m\geq3, n\geq2$). By the constant angle theorem of de Lima and Pipoli \cite{DP}, the product angle function $C$ is 
constant. In the case $C=0$, the ordinary normal Jacobi equation converts the constancy of the mean curvatures of all nearby parallel hypersurfaces into constancy of the power sums of the principal curvatures, and Newton's identities then give spectral rigidity. The case 
$0 <|C|< 1$ is subtler because the two hyperbolic factors evolve with different normal speeds and the shape operator need not preserve the factor splitting. We overcome this difficulty by coupling the normal flow with the geodesic flow generated by the zero principal curvature direction $V$. The resulting two-parameter Jacobi determinant satisfies a separation identity (see Proposition \ref{prop:two-parameter-Jacobi}). After an exponential change of variables, this determinant is encoded
by a bivariate polynomial $Q_o(x,y)$; the self-adjointness of the shape operator implies that $Q_o$ is zero-free on the product of the upper half-planes. This stability property, together
with the separation identity, forces $Q_o$ to be a monomial and therefore forces the principal curvature spectrum to be pointwise contained in a fixed finite set. Connectedness then yields constant principal curvatures (see Proposition \ref{prop:constant-angle-isoparametric}). 
\end{remark}
	
\begin{remark}
As a consequence of Theorem \ref{thm:isoparametric-HH}, every connected isoparametric hypersurface in $\mathbb{H}^m\times \mathbb{H}^n$ ($m\geq3, n\geq2$) is locally homogeneous. Moreover, it has at most three distinct constant principal curvatures. 
Since $\mathbb{H}^m\times \mathbb{H}^1$ is locally isometric to $\mathbb{H}^m\times \mathbb{R}$, the classification of isoparametric hypersurfaces in $\mathbb{H}^m\times \mathbb{H}^1$ follows from \cite{dP24}. The corresponding results in  $\mathbb{H}^2\times\mathbb{H}^2$ were obtained by \cite{GMY24}.
\end{remark}
	
The paper is organized as follows. In Section \ref{sect:2}, we collect some basic properties of $\mathbb{H}^m\times \mathbb{H}^n$ ($m\geq3, n\geq2$) and some preliminaries of the geometry of hypersurfaces of $\mathbb{H}^m\times \mathbb{H}^n$. 
In Section \ref{sect:3}, we introduce the canonical hypersurfaces, and we also construct  hypersurfaces of $\mathbb{H}^m\times \mathbb{H}^n$ with two nonconstant principal curvatures. In Section \ref{sect:4}, we establish Cartan's formulas for hypersurfaces with constant principal curvatures and constant product angle function. In Section \ref{sect:5}, we classify hypersurfaces of $\mathbb{H}^m\times \mathbb{H}^n$ ($m\geq3, n\geq2$) with at most two distinct constant principal curvatures. In Section \ref{sect:6}, we classify hypersurfaces of $\mathbb{H}^m\times \mathbb{H}^n$ ($m\geq3, n\geq2$) with constant principal curvatures and constant product angle function $C$. Section \ref{sect:7} is dedicated to the proof of Theorem \ref{thm:isoparametric-HH}. 
	
\vskip 1mm 
	
\textbf{Acknowledgments:}
H. Li and R. Tan were supported by NSFC Grant No. 12471047. Z. Yao was supported by NSFC Grant No. 12401061. 
	
\section{Preliminaries}\label{sect:2}
\subsection{The geometric structure on $\mathbb{H}^{m}\times \mathbb{H}^{n}$}\label{sect:2.1}
Let $\mathbb{R}_{1}^{k+1}$ be the $(k+1)$-dimensional Minkowski space with the Lorentzian metric $\langle\cdot,\cdot\rangle$. The hyperbolic space of curvature $-1$ can be defined as the following subset of $\mathbb{R}_{1}^{k+1}$:
$$
\mathbb{H}^{k}=\left\{\left(x_{1}, x_{2},\ldots, x_{k+1}\right) \in \mathbb{R}_{1}^{k+1} \mid - x_{1}^{2}+x_{2}^{2} + \cdots + x_{k+1}^{2} = -1 , x_{1} > 0\right\}.
$$
	
Throughout the paper we will consider $\mathbb{H}^m \times  \mathbb{H}^n$ ($m\geq3, n\geq2$) as embedded naturally in $\mathbb{R}_{1}^{m+1} \times \mathbb{R}_{1}^{n+1}$, with the induced Riemannian product metric which we also denote by $\langle\cdot,\cdot\rangle$.
	
The isometry group of $\mathbb{H}^{m}\times \mathbb{H}^{n}$ is 
$$
{\rm Iso}\left(\mathbb{H}^{m}\times \mathbb{H}^{n}\right)=\left\{\left(\begin{array}{cc}
A_{1} & 0 \\
0 & A_{2}
\end{array}\right) \mid A_{1}\in \mathrm{O}^{+}(1,m), A_{2} \in \mathrm{O}^{+}(1,n)\right\},\ \ m\neq n,
$$
	
$$
{\rm Iso}\left(\mathbb{H}^{n}\times \mathbb{H}^{n}\right)=\left\{\left(\begin{array}{cc}
A_{1} & 0 \\ 
0 & A_{2}
\end{array}\right),\left(\begin{array}{cc}
0 & B_1 \\
B_{2} & 0
\end{array}\right) \mid A_{1}, A_{2}, B_1, B_{2} \in \mathrm{O}^{+}(1,n)\right\},
$$
where $\mathrm{O}^{+}(1,m), \mathrm{O}^{+}(1,n)$ denote the orthochronous Lorentz groups.
	
The product structure $P$ on $\mathbb{H}^m\times\mathbb{H}^n$ is defined by $P: T(\mathbb{H}^m\times\mathbb{H}^n) \rightarrow T(\mathbb{H}^m\times\mathbb{H}^n)$ such that
$$
P(v_1,v_2)=(v_1,-v_2), \quad \forall\, v_1\in T\mathbb{H}^m, v_2\in T\mathbb{H}^n.
$$
Obviously, we have $P^2=\mathrm{id}$ and 
\begin{equation}\label{eqn:2.1}
\langle PX, Y\rangle=\langle X,PY\rangle, \quad \forall\,X,Y\in T(\mathbb{H}^m\times\mathbb{H}^n).
\end{equation}
Moreover, $\bar{\nabla} P=0$, where $\bar{\nabla}$ is the Levi-Civita connection on $\mathbb{H}^m\times\mathbb{H}^n$.
	
The curvature tensor $\bar{R}$ of $\mathbb{H}^m\times\mathbb{H}^n$
with the Riemannian product metric is given by
\begin{align*}
\langle\bar{R}(X,Y)Z,W\rangle=-\frac{1}{2}&\Big\{\langle Y,Z\rangle\langle X,W\rangle
-\langle X,Z\rangle\langle Y,W\rangle\\
&+\langle PY,Z\rangle\langle PX,W\rangle-\langle PX,Z\rangle\langle PY,W\rangle\Big\},
\end{align*}
where $X, Y, Z, W\in T(\mathbb{H}^m\times\mathbb{H}^n)$. 
	
\subsection{Hypersurfaces of $\mathbb{H}^{m}\times \mathbb{H}^{n}$}\label{sect:2.2}
	
Let $M$ be an orientable hypersurface of $\mathbb{H}^m \times \mathbb{H}^n$ with $N$ a unit normal vector field. The induced metric on $M$ is still denoted by $\langle\cdot,\cdot\rangle$. Then, with respect to the product structure $P$, the product angle function $C: M\rightarrow\mathbb{R}$ and a vector field $V$ tangent to $M$ are defined by
\begin{align*}
C&:=\langle PN,N\rangle,\\
V&:=PN-CN.
\end{align*}
It is clear that $-1\leq C\leq 1$ and $\|V\|^2:=\langle V,V\rangle=1-C^2$.
	
Let $T: TM\rightarrow TM$ be the tangential component of the restriction of $P$ to $M$, 
i.e.,
$$
TX=P X-\langle PX,N\rangle N=P X-\langle X,V\rangle N
$$
for any tangent vector field $X$ of $M$. 
Let $V^\perp\subset TM$ denote the orthogonal complement distribution of $V$, 
then $V^\perp$ is $T$-invariant and $T|_{V^\perp}$ is an orthogonal involution, 
i.e., $(T|_{V^\perp})^2 = \operatorname{id_{V^\perp}}$. 
	
Let $\nabla$ be the Levi-Civita connection of the induced metric on $M$. 
The Gauss and Weingarten formulas are
\begin{align*}
\bar{\nabla}_X Y=\nabla_X Y+\langle AX,Y\rangle N, \quad \bar{\nabla}_XN=-AX,
\end{align*}
where $A$ is the shape operator of $M$.
	
Now, the Gauss and Codazzi equations of $M$ are given by
\begin{equation}\label{eqn:2.2}
\begin{aligned}
R(X, Y)Z=&-\frac{1}{2}\Big(\langle Y,Z\rangle X-\langle X,Z\rangle Y+\langle TY,Z\rangle TX
-\langle TX,Z\rangle TY\Big)\\
&+\langle AY,Z\rangle AX-\langle AX,Z\rangle AY,
\end{aligned}
\end{equation}
\begin{equation}\label{eqn:2.3}
(\nabla_XA)Y-(\nabla_Y A)X=-\frac{1}{2}\Big(\langle X,V\rangle TY-\langle Y,V\rangle TX\Big),
\end{equation}
where $X, Y, Z\in TM$, and $R$ denotes the curvature tensor of $M$ with respect to the metric $\langle\cdot,\cdot\rangle $.
	
Notice that the product structure $P$ of $\mathbb{H}^m \times \mathbb{H}^n$ satisfies \eqref{eqn:2.1} and $\bar{\nabla}P=0$. Then, we can obtain the following lemma which describes some properties of the function $C$ and the vector field $V$.
	
\begin{lemma}\label{lemma:2.1}
Let $M$ be an orientable hypersurface of $\mathbb{H}^m \times \mathbb{H}^n$ and $A$ the shape operator associated to the unit normal field $N$. Then the gradient of $C$ and the covariant derivative of $V$ are given by
\begin{equation}\label{eqn:2.4}
\nabla C=-2 AV, \quad \nabla_{X} V=C AX -T A X, \ \ \forall\ X\in TM.
\end{equation}
\end{lemma}
	
\begin{proof}
By the definition of the product angle function $C=\langle PN,N\rangle$, we have
$$
\begin{aligned}
XC&=X\langle PN,N\rangle=\langle P\bar{\nabla}_X N,N\rangle+\langle PN,\bar{\nabla}_X N\rangle\\
&=-2\langle AX,V\rangle=-2\langle AV,X\rangle, \ \ \forall\ X\in TM.
\end{aligned}
$$
It follows that $\nabla C=-2 AV$.
		
Then, using $V=PN-CN$ and $\nabla C=-2 AV$, we obtain
$$
\begin{aligned}
\nabla_X V&=\bar{\nabla}_X V-\langle AX,V\rangle N=\bar{\nabla}_X (PN-CN)-\langle AX,V\rangle N\\
&=P\bar{\nabla}_X N-(XC)N-C\bar{\nabla}_X N-\langle AX,V\rangle N\\
&=-PAX+CAX+\langle AX,V\rangle N=C AX -T A X,
\end{aligned}
$$
for any tangent vector field $X$ of $M$.
\end{proof}
	
\begin{corollary}\label{coro:2.2}
When the product angle function $C$ is constant on $M$ and $C\neq \pm1$, we have
\begin{enumerate}
\item $V$ is a principal vector field of $M$, and it satisfies $AV=0$ and $\nabla_V V = 0$.
\item $V^\perp$ is $A$-invariant, i.e., $A(V^\perp) \subset V^\perp$. 
\end{enumerate}
\end{corollary}

\subsection{The two-parameter Jacobi field formula}\label{sect:2.3}
	
Let $M$ be a connected oriented hypersurface of $\mathbb H^m\times\mathbb H^n$ ($m\geq3, n\geq2$) with constant product angle function $C$ satisfying $|C|<1$. Put
$$
\alpha=\sqrt{\frac{1+C}{2}},\qquad \qquad
\beta=\sqrt{\frac{1-C}{2}}. 
$$
Then $\alpha,\beta>0$, $\alpha^2+\beta^2=1$. Moreover $\alpha=\beta$ if and only if $C=0$. Since $|V|=2\alpha\beta$, the vector field $\bar{V}=\frac{V}{2\alpha\beta}$ 
is unit. By Corollary \ref{coro:2.2}, $A\bar{V}=0$. Locally, we may write 
$$
N=(\alpha\nu_1,\beta\nu_2),\qquad
\bar{V}=(\beta\nu_1,-\alpha\nu_2), 
$$
where $\nu_1$ and $\nu_2$ are unit vector fields in the two hyperbolic factors. Define
$$
\mathcal D_+=\{(X_1,0)\in T(\mathbb{H}^m\times \mathbb{H}^n)|\,X_1\perp\nu_1\},\qquad
\mathcal D_-=\{(0,X_2)\in T(\mathbb{H}^m\times \mathbb{H}^n)|\,X_2\perp\nu_2\}.
$$
Thus
$$
TM=\mathbb R\bar{V}\oplus\mathcal D_+\oplus\mathcal D_-,
\qquad \dim\mathcal D_+=m-1,
\qquad \dim\mathcal D_-=n-1.
$$
Since $A$ is self-adjoint and $A\bar{V}=0$, the distribution $\bar{V}^\perp$ is $A$-invariant. Relative to $\bar{V}^\perp=\mathcal D_+\oplus\mathcal D_-$, we write
$$
S=A|_{\bar{V}^\perp}=
\begin{pmatrix}
B&K\\
K^{\mathsf t}&D
\end{pmatrix},
\qquad B^{\mathsf t}=B,
\qquad D^{\mathsf t}=D.
$$
	
Now, we derive a two-parameter Jacobi field formula valid for every hypersurface with constant product angle function $|C|<1$. 
	
\begin{proposition}\label{prop:two-parameter-Jacobi}
Let $M$ be a connected oriented hypersurface of $\mathbb H^m\times\mathbb H^n$ with constant product angle function $|C|<1$. Let $\psi_s$ be the local flow of $\bar{V}$, and $\Phi_t(M)$ be a nearby parallel hypersurface of $M$ defined by 
\begin{equation}\label{eqn:parallel}
\begin{aligned}
\Phi_t:\ M &\longrightarrow \mathbb{H}^m\times \mathbb{H}^n,\\
x&\longmapsto {\rm exp}_x(tN_x),  
\end{aligned}
\end{equation}
where ${\rm exp}$ denotes the exponential map in 
$\mathbb{H}^m\times \mathbb{H}^n$, $N_x$ is the unit normal vector at $x\in M$. Let $\Psi_{t,s}=\Phi_t\circ\psi_s$ 
and fix $o\in M$. Then after parallel translation along the relevant factor geodesics, the restriction of $d\Psi_{t,s}$ to $\bar{V}_o^\perp$ is represented by
\begin{equation}\label{eqn:7.1a1}
J_o(t,s)=
\begin{pmatrix}
\cosh r\,I_{m-1}&0\\
0&\cosh w\,I_{n-1}
\end{pmatrix}
-\begin{pmatrix}
\dfrac{\sinh r}{\alpha}I_{m-1}&0\\
0&\dfrac{\sinh w}{\beta}I_{n-1}
\end{pmatrix}S_o,
\end{equation}
where
$$
r=\alpha t+\beta s,
\qquad
w=\beta t-\alpha s.
$$
Moreover, $d\Psi_{t,s}(\bar{V}_o)$ is a unit vector orthogonal to $d\Psi_{t,s}(\bar{V}_o^\perp)$. Consequently, if
$d_o(t,s):=\det J_o(t,s)$, then $\det d\Psi_{t,s}=d_o(t,s)$ holds at $o\in M$. 
		
Furthermore, if $M$ is isoparametric, then at $o\in M$, it holds 
\begin{equation}\label{eqn:7.1dd1}
\det d\Psi_{t,s}=d_o(t,s)=f(t)g_o(s),\quad f(t):=d_o(t,0), \ \ g_o(s):=d_o(0,s). 
\end{equation}
This scalar determinant identity \eqref{eqn:7.1dd1}, initially valid for sufficiently small real $t$ and $s$, extends to all $(t,s)\in\mathbb C^2$. 
\end{proposition}
	
\begin{proof}
Since $C$ is constant, Corollary \ref{coro:2.2} gives $A\bar{V}=0$ and $\nabla_{\bar{V}}{\bar{V}}=0$. The Gauss and Weingarten formulas therefore yield
$$
\bar\nabla_{\bar{V}}{\bar{V}}=0,
\qquad
\bar\nabla_{\bar{V}}N=0.
$$
Hence the integral curves of $\bar{V}$ are ambient geodesics and $N$ is parallel along them.
		
For $X\in \bar{V}^\perp$, differentiation of the expression for $N$ gives
$$
\bar\nabla_X(\nu_1,0)=-\frac{1}{\alpha}\big((SX)_1,0\big), 
\qquad
\bar\nabla_X(0,\nu_2)=-\frac{1}{\beta}\big(0,(SX)_2\big), 
$$
where $(SX)_i$ denotes the corresponding factor component. Differentiating $\bar{V}$ also gives 
$$
\bar\nabla_X \bar{V}=-\left(\frac\beta\alpha(SX)_1, -\frac\alpha\beta (SX)_2\right). 
$$
If $x=(x_1,x_2)$, the geodesic and parallelism properties above imply
$$
\Psi_{t,s}(x)=
\left(
\exp^{\mathbb{H}^m}_{x_1}(r\nu_1(x)),
\exp^{\mathbb{H}^n}_{x_2}(w\nu_2(x))
\right).
$$
For a variation with initial vector $X\in \bar{V}_x^\perp$, both $X_i$ and $(SX)_i$ are orthogonal to $\nu_i$. The Jacobi fields in the two hyperbolic factors therefore have initial data
$$
X_1,\ \ -\frac{1}{\alpha}(SX)_1
\qquad\text{and}\qquad
X_2,\ \ -\frac{1}{\beta}(SX)_2, 
$$
respectively. The orthogonal Jacobi equation in hyperbolic space is $Y''-Y=0$. Then by a standard and straightforward computation of Jacobi field theory (cf. \cite[Sec. 8.2]{BCO}), we obtain the displayed matrix $J_o(t,s)$ at $o\in M$, see \eqref{eqn:7.1a1}. Varying the initial point in the $\bar{V}$-direction only changes the flow parameter $s$, hence $d\Psi_{t,s}(\bar{V}_o)$ is the parallel translate of $\bar{V}_{\psi_s(o)}$ along the normal geodesic through $\psi_s(o)$,  and it follows that $d\Psi_{t,s}(\bar{V}_o)$ is a unit vector orthogonal to  $d\Psi_{t,s}(\bar{V}_o^\perp)$. Thus the full Jacobian of $\Psi_{t,s}$ equals $d_o(t,s)$ at $o\in M$. 
		
In the following, if $M$ is isoparametric, it remains to prove the determinant identity. Let $j_t(x)$ denote the oriented Jacobian of $\Phi_t$ at $x$. For sufficiently small real $t$,
\begin{equation}\label{eqn:7.1ddaa1}
\frac{d}{dt}\log j_t(x)
=-\operatorname{tr}A_t(\Phi_t(x)),
\qquad
j_0(x)=1.
\end{equation}
Here $A_t$ is the shape operator of the parallel hypersurface $\Phi_t(M)$ with
respect to the unit normal vector field $\frac{d{\rm exp}_{x}(tN_{x})}{dt}$,  where $x\in M$. Since $M$ is isoparametric, the right-hand side of \eqref{eqn:7.1ddaa1} is independent of $x$ on a sufficiently small connected neighborhood. Thus $j_t(x)$ is independent of $x$ there. Since the full Jacobian of $\Psi_{t,s}$ equals $d_o(t,s)$ at $o\in M$, the composition rule for Jacobians gives
$$
\begin{aligned}
d_o(t,s)
&=j_t(\psi_s(o))\det(d\psi_s)_o\\
&=j_t(o)\det(d\psi_s)_o
=d_o(t,0)d_o(0,s).
\end{aligned}
$$
For fixed $o\in M$, the three expressions $d_o(t,s)$, $f(t)$ and $g_o(s)$ are entire functions of the indicated complex variables, because they are finite determinants of hyperbolic functions. Applying the identity theorem first in $t$ and then in $s$ proves the complex extension. 
\end{proof}
	
\begin{proposition}\label{pro:7.1a1}
Let $M$ be a connected oriented hypersurface of $\mathbb{H}^m\times \mathbb{H}^n$ ($m\geq3, n\geq2$) with product angle function $C\equiv0$,   and $\Phi_t(M)$ be the parallel hypersurface of $M$. Then at any point $x\in M$, for any $X\in \bar{V}_x^\perp$ such that  $AX=\lambda(x) X$, we have 
\begin{equation}\label{eqn:7.1a2}
d\Phi_t(X)=
\left(
\cosh\frac{t}{\sqrt2}
-\sqrt2\lambda(x)\sinh\frac{t}{\sqrt2}
\right)\mathcal P_tX,
\qquad
d\Phi_t(\bar{V}_x)=\mathcal P_t(\bar{V}_x),
\end{equation}
where $\mathcal P_t$ denotes parallel translation along the normal geodesic through $x$. 
Let $A_t$ be the shape operator of the parallel hypersurface $\Phi_t(M)$, then we have 
\begin{equation}\label{eqn:7.1a3}
A_t\mathcal P_t(X)=-\frac1{\sqrt2}\frac{\sinh\frac t{\sqrt2}-\sqrt2\,\lambda(x)\cosh\frac t{\sqrt2}}{\cosh\frac t{\sqrt2}-\sqrt2\,\lambda(x)\sinh\frac t{\sqrt2}}\mathcal P_t(X), 
\quad A_t\mathcal P_t(\bar{V}_x)=0. 
\end{equation}
\end{proposition}
	
\begin{proof}
When $C=0$, one has $\alpha=\beta=1/\sqrt2$. Substituting $s=0$ into 
Proposition \ref{prop:two-parameter-Jacobi}, we get \eqref{eqn:7.1a2}. Then 
by using $\bar{\nabla}_{\partial_t} d\Phi_t(X)=-A_t(d\Phi_t(X))$, 
we further obtain \eqref{eqn:7.1a3}. 
\end{proof}

\section{Canonical examples}\label{sect:3}
	
In this section, we introduce some canonical examples of hypersurfaces in $\mathbb{H}^m \times \mathbb{H}^n$, $m \geq 3$, $n \geq 2$. 
	
\begin{lemma}[\cite{DP}]\label{lemma:3.5}
Let $M$ be a hypersurface of $\mathbb{H}^m\times\mathbb{H}^n$ ($m\geq3, n\geq2$) with $C^2=1$. Then, up to isometries of $\mathbb{H}^m \times \mathbb{H}^n$, $M$ is either an open part of hypersurface $\Sigma\times\mathbb{H}^n$, or an open part of hypersurface $\mathbb{H}^m \times \tilde{\Sigma}$, where $\Sigma$ and $\tilde \Sigma$ are hypersurfaces of $\H^m$ and $\H^n$, respectively.
\end{lemma}
	
It is obvious that the hypersurface $\Sigma\times\mathbb{H}^n$ has constant product angle function $C=1$, and its principal curvatures are $\kappa^{\Sigma}_1,...,\kappa^{\Sigma}_{m-1}$, $0$ (multiplicity $n$), where $\kappa^{\Sigma}_1 , ... , \kappa^{\Sigma}_{m-1}$ are principal curvatures of $\Sigma \hookrightarrow \mathbb{H}^m$. On the other hand, from the expression of geodesics in  $\mathbb{H}^m\times\mathbb{H}^n$, we know that the parallel hypersurface of $\Sigma\times\mathbb{H}^n$ at distance $l$ is $\Gamma\times \mathbb{H}^n$, where $\Gamma$ is a parallel hypersurface of $\Sigma$ at distance $l$ in $\mathbb{H}^m$. 
Together with Cartan's classification in hyperbolic space, these observations yield the following lemma.
	
\begin{lemma}\label{lemma:3.2}
For any hypersurface $\Sigma\times\mathbb{H}^n$ of $\mathbb{H}^m\times\mathbb{H}^n$,
the following statements are equivalent:
\begin{enumerate}
\item[(1)] $\Sigma\times\mathbb{H}^n$ has constant principal curvatures; 
\item[(2)] $\Sigma\times\mathbb{H}^n$ is an isoparametric hypersurface; 
\item[(3)] $\Sigma$ is a hypersurface in $\mathbb{H}^m$ with constant principal curvatures. 
\end{enumerate}
\end{lemma}
	
Similarly, we have the following lemma:
	
\begin{lemma}\label{lemma:3.2aa}
For any hypersurface $\mathbb{H}^m\times\tilde{\Sigma}$ of $\mathbb{H}^m\times\mathbb{H}^n$, the following statements are equivalent:
\begin{enumerate}
\item[(1)]$\mathbb{H}^m\times\tilde{\Sigma}$ has constant principal curvatures; 
\item[(2)]$\mathbb{H}^m\times\tilde{\Sigma}$ is an isoparametric hypersurface; 
\item[(3)]$\tilde{\Sigma}$ is a hypersurface in $\mathbb{H}^n$ with constant principal curvatures. 
\end{enumerate}
\end{lemma}
	
Now, we have the following examples with constant principal curvatures and $C^2\equiv1$. 
	
\begin{example}\label{exam:1}
For any smooth hypersurface $\Sigma$ of $\mathbb{H}^m$ with constant principal curvatures, one can define a hypersurface of $\mathbb{H}^m\times\mathbb{H}^n$ by
$$
\Sigma\times\mathbb{H}^n:=\left\{(x,y)\in \mathbb{H}^m\times\mathbb{H}^n~|~x\in \Sigma,\ y\in\mathbb{H}^n\right\},
$$
which has at most three distinct constant principal curvatures and $C\equiv1$. 
\end{example}

\begin{example}\label{exam:2}
For any smooth hypersurface $\tilde{\Sigma}$ of $\mathbb{H}^n$ with constant principal curvatures, one can define a hypersurface of $\mathbb{H}^m\times\mathbb{H}^n$ by
$$
\mathbb{H}^m\times\tilde{\Sigma}:=\left\{(x,y)\in \mathbb{H}^m\times\mathbb{H}^n ~|~x\in\mathbb{H}^m,\ y\in \tilde{\Sigma}\right\}, 
$$
which has at most three distinct constant principal curvatures and $C\equiv-1$. 
\end{example}
	
Next, inspired by \cite{GMY24}, for any two smooth hypersurfaces of $\mathbb{H}^m$ and $\mathbb{H}^n$, we can construct a hypersurface of $\mathbb{H}^m\times\mathbb{H}^n$ with constant product angle function $C$. 
	
\begin{example}\label{exam:3.4a1a}
For any constant $-1<C<1$, and any two smooth hypersurfaces $\Sigma_1$ of $\mathbb{H}^m$ and $\Sigma_2$ of $\mathbb{H}^n$, consider the map
$\Phi:I\times\Sigma_1 \times  \Sigma_2\rightarrow\mathbb{H}^m\times\mathbb{H}^n$:
$(t,x,y)\rightarrow(p(t,x),q(t,y))$, where
\begin{equation}\label{eqn:3.1}
\begin{aligned}
&p(t,x)=\cosh(\sqrt{\frac{1-C}{2}}t)x+\sinh(\sqrt{\frac{1-C}{2}}t)N(x),\\
&q(t,y)=\cosh(\sqrt{\frac{1+C}{2}}t)y+\sinh(\sqrt{\frac{1+C}{2}}t)\tilde{N}(y),
\end{aligned}	
\end{equation}
$N(x)$ and $\tilde{N}(y)$ are unit normal vector fields of $\Sigma_1 \hookrightarrow \mathbb{H}^m$ and $\Sigma_2 \hookrightarrow \mathbb{H}^n$, respectively. Restricting $\Phi$ to its regular set gives a hypersurface with constant product angle function $C$. We denote the hypersurface constructed in \eqref{eqn:3.1} by $M_{\Sigma_1,\Sigma_2}^C$.
		
The unit normal vector field $N$ of $M_{\Sigma_1,\Sigma_2}^C$ is given by 
\begin{equation}\label{eqn:3.2}
\begin{aligned}
N&=(N_1,N_2)=(\sqrt{\frac{1+C}{1-C}}\frac{\partial p(t,x)}{\partial t},
-\sqrt{\frac{1-C}{1+C}}\frac{\partial q(t,y)}{\partial t})\\
&=(\sqrt{\frac{1+C}{2}}\sinh(\sqrt{\frac{1-C}{2}}t)x
+\sqrt{\frac{1+C}{2}}\cosh(\sqrt{\frac{1-C}{2}}t)N(x),\\
&\ \ \ \ -\sqrt{\frac{1-C}{2}}\sinh(\sqrt{\frac{1+C}{2}}t)y
-\sqrt{\frac{1-C}{2}}\cosh(\sqrt{\frac{1+C}{2}}t)\tilde{N}(y)).
\end{aligned}
\end{equation}
It follows that $M_{\Sigma_1,\Sigma_2}^C$ has constant product angle function $C$. 
		
Let $\kappa_1(x) , ... , \kappa_{m-1}(x)$ and $\tilde{\kappa}_1(y) , ... , \tilde{\kappa}_{n-1}(y)$ be the principal curvatures of $\Sigma_1 \hookrightarrow \mathbb{H}^m$ and $\Sigma_2 \hookrightarrow \mathbb{H}^n$, respectively. Let $\{e_i\}_{i=1}^{m-1}$ and $\{\tilde{e}_j\}_{j=1}^{n-1}$ be the local orthonormal principal frames of $\Sigma_1 \hookrightarrow \mathbb{H}^m$ and $\Sigma_2 \hookrightarrow \mathbb{H}^n$, respectively. At any $x\in \Sigma_1$ and $y\in \Sigma_2$, we assume that $A_1e_i=\kappa_i(x)e_i$ and $A_2\tilde{e}_j=\tilde{\kappa}_j(y)\tilde{e}_j$, where $A_1$ and $A_2$ are the shape operators of $\Sigma_1 \hookrightarrow \mathbb{H}^m$ and $\Sigma_2 \hookrightarrow \mathbb{H}^n$, respectively. Then we have 
$$
\begin{aligned}
&d\Phi(e_i)=\Big(\cosh(\sqrt{\frac{1-C}{2}}t)-\sinh(\sqrt{\frac{1-C}{2}}t)\kappa_i(x)\Big)e_i,\\
&d\Phi(\tilde{e}_j)=\Big(\cosh(\sqrt{\frac{1+C}{2}}t)-\sinh(\sqrt{\frac{1+C}{2}}t)\tilde{\kappa}_j(y)\Big)\tilde{e}_j.
\end{aligned}
$$
		
A direct calculation gives $AV=0$ and 
\begin{equation}\label{eqn:3.3}
\begin{aligned}
&A(d\Phi(e_i))=-\sqrt{\frac{1+C}{2}}\frac{\sinh(\sqrt{\frac{1-C}{2}}t)-\cosh(\sqrt{\frac{1-C}{2}}t)\kappa_i(x)}
{\cosh(\sqrt{\frac{1-C}{2}}t)-\sinh(\sqrt{\frac{1-C}{2}}t)\kappa_i(x)}d\Phi(e_i),\ \ 1\leq i\leq m-1,\\ 
& A(d\Phi(\tilde{e}_j))=\sqrt{\frac{1-C}{2}}\frac{\sinh(\sqrt{\frac{1+C}{2}}t)-\cosh(\sqrt{\frac{1+C}{2}}t)\tilde{\kappa}_j(y)}
{\cosh(\sqrt{\frac{1+C}{2}}t)-\sinh(\sqrt{\frac{1+C}{2}}t)\tilde{\kappa}_j(y)}d\Phi(\tilde{e}_j),\ \ 1\leq j\leq n-1. 
\end{aligned}
\end{equation}
\end{example}
	
\begin{remark}\label{rem:3.1}
According to the expressions \eqref{eqn:3.3}, we see that, for general hypersurfaces $\Sigma_1 \hookrightarrow \mathbb{H}^m$ and $\Sigma_2 \hookrightarrow \mathbb{H}^n$, hypersurface $M_{\Sigma_1, \Sigma_2}^C$ has nonconstant principal curvatures and constant product angle function $C$. For example, if we choose 
$\Sigma_1$ and $\Sigma_2$ as two (non-horospherical) totally umbilical hypersurfaces, then $M_{\Sigma_1, \Sigma_2}^C$ has at most three 
principal curvatures which depend on the parameter $t$. For the same reason as pointed out in Remark 3.6 of \cite{GMY24}, $M_{\Sigma_1, \Sigma_2}^C$ has constant principal curvatures if and only if, up to isometries of $\mathbb{H}^m \times \mathbb{H}^n$, $\kappa_i(x) = \tilde{\kappa}_j(y)=1$ or $\kappa_i(x) = -\tilde{\kappa}_j(y)=1$.
\end{remark}
	
\begin{example}\label{exam:3}
For any given $-1<C<1$, we choose two horospheres $\Sigma_1 \hookrightarrow \mathbb{H}^m$ and $\Sigma_2 \hookrightarrow \mathbb{H}^n$ in \eqref{eqn:3.1}. 
Let $N(x)$ and $\tilde{N}(y)$ be the unit normal vector fields of $\Sigma_1 \hookrightarrow \mathbb{H}^m$ and  $\Sigma_2 \hookrightarrow \mathbb{H}^n$ respectively, such that the principal curvatures of $\Sigma_1 \hookrightarrow \mathbb{H}^m$ and $\Sigma_2 \hookrightarrow \mathbb{H}^n$ are $\kappa_i(x)=1$ ($1\leq i\leq m-1$) and $\tilde{\kappa}_j(y)=-1$ ($1\leq j\leq n-1$), respectively. In this case, we call the hypersurface $M_{1,-1}^{C}$.
		
Now, by \eqref{eqn:3.3}, when $C\neq 0$, $M_{1,-1}^{C}$ has three distinct constant principal curvatures 
$$
\begin{tabular}{|c|c|c|c|}
\hline
{\rm value} & $0$ & $\sqrt{\frac{1+C}{2}}$  & $\sqrt{\frac{1-C}{2}}$ \\
\hline
{\rm multiplicity} & $1$ & $m-1$ & $n-1$ \\ 
\hline
\end{tabular}
$$
For principal curvatures $\sqrt{\frac{1+C}{2}}$ and $\sqrt{\frac{1-C}{2}}$, let $V_{\sqrt{\frac{1+C}{2}}}$ and $V_{\sqrt{\frac{1-C}{2}}}$ be the corresponding eigenspaces. Then, for any $X\in V_{\sqrt{\frac{1+C}{2}}}$, it holds that $PX=X$.  
For any $X\in V_{\sqrt{\frac{1-C}{2}}}$, it holds that $PX=-X$. 

When $C=0$, hypersurface $M_{1,-1}^{0}$ has two distinct constant principal curvatures $0$ and $\frac{1}{\sqrt{2}}$. 
		
\end{example}
	
\begin{example}\label{exam:4}
For any given $-1<C<1$, we choose two horospheres $\Sigma_1 \hookrightarrow \mathbb{H}^m$ and  $\Sigma_2 \hookrightarrow \mathbb{H}^n$ in \eqref{eqn:3.1}. 
Let $N(x)$ and $\tilde{N}(y)$ be the unit normal vector fields of $\Sigma_1 \hookrightarrow \mathbb{H}^m$ and $\Sigma_2 \hookrightarrow \mathbb{H}^n$ respectively, such that the principal curvatures of $\Sigma_1 \hookrightarrow \mathbb{H}^m$ and $\Sigma_2 \hookrightarrow \mathbb{H}^n$ are $\kappa_i(x)=1$ ($1\leq i\leq m-1$) and $\tilde{\kappa}_j(y)=1$ ($1\leq j\leq n-1$), respectively. 
In this case, we call the hypersurface $M_{1,1}^{C}$. 
		
Now, by \eqref{eqn:3.3}, hypersurface $M_{1,1}^{C}$ has three distinct constant principal curvatures 
$$
\begin{tabular}{|c|c|c|c|} 
\hline
{\rm value} & $0$ & $\sqrt{\frac{1+C}{2}}$  & $-\sqrt{\frac{1-C}{2}}$\\
\hline
{\rm multiplicity} & $1$ & $m-1$ & $n-1$\\ 
\hline
\end{tabular}
$$
		
For principal curvatures $\sqrt{\frac{1+C}{2}}$ and $-\sqrt{\frac{1-C}{2}}$, 
let $V_{\sqrt{\frac{1+C}{2}}}$ and $V_{-\sqrt{\frac{1-C}{2}}}$ be the corresponding  eigenspaces. Then, for any $X\in V_{\sqrt{\frac{1+C}{2}}}$, it holds that $PX=X$.  
For any $X\in V_{-\sqrt{\frac{1-C}{2}}}$, it holds that $PX=-X$. 
		
When $C =-\frac{(m-n)(m+n-2)}{2+(m-2)m+(n-2) n}$, hypersurface $M_{1,1}^{C}$ is a minimal hypersurface. 
\end{example}
	
\begin{theorem}[\cite{Toj26}]\label{thm:Toj26}
Let $M$ be a hypersurface of $\mathbb{H}^m\times \mathbb{H}^n$ with product angle function $|C|<1$. If $AT = TA$, then $M$ is locally given by
$$
\tilde{\Phi}:I\times \Sigma_1\times \Sigma_2 \hookrightarrow \H^m\times \H^n,\quad (t,x,y)\rightarrow(p(t,x),q(t,y))
$$
where
\begin{equation}\label{eqn:ATTA}
\begin{aligned}
&p(t,x)=\cosh(a(t))x+\sinh(a(t))N(x),\\
&q(t,y)=\cosh(b(t))y+\sinh(b(t))\tilde{N}(y),
\end{aligned}	
\end{equation}
$N(x)$ and $\tilde{N}(y)$ are unit normal vector fields of $\Sigma_1 \hookrightarrow \mathbb{H}^m$ and $\Sigma_2 \hookrightarrow \mathbb{H}^n$, respectively. The functions $a(t)$ and $b(t)$ are smooth functions satisfying 
$$
a(0) = b(0) = 0,\quad a'(t), b'(t)> 0 \ \ \text{and}\ \ (a'(t))^2 + (b'(t))^2 = 1. 
$$ 
\end{theorem}
	
\begin{proposition}\label{prop:ATTA}
Let $M$ be a hypersurface of $\mathbb{H}^m\times \mathbb{H}^n$ satisfying 
$AT = TA$. If $M$ has constant principal curvatures and constant product angle function $C\neq \pm 1$, then $M$ is either an open part of $M_{1,-1}^{C}$ for some $C\in(-1,1)$, or an open part of $M_{1,1}^{C}$ for some $C\in (-1,1)$. 
\end{proposition}
	
\begin{proof}
According to Theorem \ref{thm:Toj26}, $M$ is locally given by the immersion \eqref{eqn:ATTA}. By Corollary 5.1 of 
\cite{Toj26}, it follows that $M$ can be further locally given by the immersion \eqref{eqn:3.1}. 
Thus by Remark \ref{rem:3.1}, $M$ is an open part of one of the hypersurfaces in Examples \ref{exam:3} and \ref{exam:4}.  
\end{proof}
	
In the following, we construct hypersurfaces of $\mathbb{H}^m\times \mathbb{H}^n$ with two nonconstant principal curvatures. 
	
\begin{example}\label{ex:two-pc-horospheres}
Let $\Sigma_1\subset \mathbb{H}^m$ and $\Sigma_2\subset \mathbb{H}^n$ be horospheres. 
Let $N$ and $\widetilde N$ be unit normal vector fields of
$\Sigma_1\hookrightarrow \mathbb{H}^m$ and $\Sigma_2\hookrightarrow \mathbb{H}^n$, 
respectively, and let $A_1,A_2$ denote their shape operators. We
consider the two cases
$$
\text{(A)}\quad A_1=\mathrm{id},\quad A_2=-\mathrm{id},
\qquad\qquad
\text{(B)}\quad A_1=\mathrm{id},\quad A_2=\mathrm{id}.
$$
		
Let $a,b:I\to\mathbb R$ be smooth functions satisfying
$$
a(0)=b(0)=0,\qquad a'(t),b'(t)>0,\qquad
\bigl(a'(t)\bigr)^2+\bigl(b'(t)\bigr)^2=1,
$$
and write
$$
a'(t)=\cos\theta(t),\qquad b'(t)=\sin\theta(t),
\qquad 0<\theta(t)<\frac\pi2.
$$
Consider the immersion
$\tilde\Phi:I\times\Sigma_1\times\Sigma_2\to \H^m\times \H^n$ defined by
\eqref{eqn:ATTA}. Let $M$ be the hypersurface generated by the mapping $\tilde\Phi$. 
Then according to (3.7) of \cite{Toj26}, we know that the principal
curvatures of $M$ are as follows. 
		
\medskip
\noindent\textbf{Case (A):}
\begin{equation}\label{eq:pcs-caseA}
-\theta'(t),\qquad \sin\theta(t),\qquad \cos\theta(t),
\end{equation}
with multiplicities $1$, $m-1$ and $n-1$, respectively. 
		
\medskip
\noindent\textbf{Case (B):}
\begin{equation}\label{eq:pcs-caseB}
-\theta'(t),\qquad \sin\theta(t),\qquad -\cos\theta(t),
\end{equation}
with multiplicities $1$, $m-1$ and $n-1$, respectively.
		
We now choose $\theta(t)$ so that two of the three numbers in
\eqref{eq:pcs-caseA} or \eqref{eq:pcs-caseB} coincide, while none of
them is constant.
		
\medskip
\noindent\textbf{Two principal curvatures in Case (A).}
\begin{itemize}
\item[(a)] If $\theta'=-\sin\theta$, 
then the principal curvatures are 
$$
\sin\theta(t)\quad(\text{multiplicity }m),
\qquad
\cos\theta(t)\quad(\text{multiplicity }n-1).
$$
\item[(b)] If $\theta'=-\cos\theta$,
then the principal curvatures are 
$$
\sin\theta(t)\quad(\text{multiplicity }m-1),
\qquad
\cos\theta(t)\quad(\text{multiplicity }n).
$$
\end{itemize}
In both subcases, the hypersurfaces have two distinct principal
curvatures, which depend on $t$. 
		
\medskip
\noindent\textbf{Two principal curvatures in Case (B).}
\begin{itemize}
\item[(a)] If $\theta'=-\sin\theta$,
then the principal curvatures are 
$$
\sin\theta(t)\quad(\text{multiplicity }m),
\qquad
-\cos\theta(t)\quad(\text{multiplicity }n-1).
$$
\item[(b)] If $\theta'=\cos\theta$, 
then the principal curvatures are 
$$
\sin\theta(t)\quad(\text{multiplicity }m-1),
\qquad
-\cos\theta(t)\quad(\text{multiplicity }n).
$$
\end{itemize}
Again, the hypersurfaces have exactly two distinct principal curvatures, which are nonconstant. 
		
For example, in the subcase $\theta'=-\sin\theta$, the solutions with
$\theta(0)\in(0,\pi/2)$ are
$$
\theta(t)=2\arctan\bigl(ce^{-t}\bigr),\qquad 0<c<1,
$$
and the corresponding functions $a,b$ are
$$
a(t)=\int_0^t\cos\theta(s)\,ds,\qquad
b(t)=\int_0^t\sin\theta(s)\,ds.
$$
\end{example}
	
Finally, inspired by \cite{DP,GMY24}, we construct the following example. 
	
\begin{example}\label{exam:5}
For any given $\tau<-1$, we define $M_\tau=\{(p,q)\in
\mathbb{H}^n\times\mathbb{H}^n|\ \langle p,q\rangle=\tau\}$.
The unit normal vector field to $M_\tau$ in $\mathbb{H}^n \times \mathbb{H}^n$ is given by
$$
N_{(p, q)}=\frac{1}{\sqrt{2(\tau^2-1)}}(q+\tau p, p+\tau q).
$$
It follows that the product angle function $C=\langle PN,N\rangle = 0$ holds on $M_\tau$. 
		
For any $(v_1,v_2)\in T_{(p,q)}M_\tau$, the shape operator $A$ associated to $N$ takes the form
$$
A(v_1,v_2)=\frac{1}{\sqrt{2(\tau^2-1)}}\big[-(v_2,v_1)-\tau(v_1,v_2)
-\langle p,v_2\rangle(p,-q)\big].
$$
It follows that $M_{\tau}$ has three distinct constant principal curvatures 
$$
\begin{tabular}{|c|c|c|c|}
\hline
{\rm value} & $0$ & $\sqrt{\frac{\tau-1}{2(\tau+1)}}$  & $\sqrt{\frac{\tau+1}{2(\tau-1)}}$\\
\hline
{\rm multiplicity} & $1$ & $n-1$ & $n-1$\\ 
\hline
\end{tabular}
$$
For principal curvatures $\sqrt{\frac{\tau-1}{2(\tau+1)}}$ and $\sqrt{\frac{\tau+1}{2(\tau-1)}}$, let $V_{\sqrt{\frac{\tau-1}{2(\tau+1)}}}$ and $V_{\sqrt{\frac{\tau+1}{2(\tau-1)}}}$ be the corresponding eigenspaces. Then, $PV_{\sqrt{\frac{\tau-1}{2(\tau+1)}}}=V_{\sqrt{\frac{\tau+1}{2(\tau-1)}}}$.  
		
The tube of radius $l$ over the submanifold $\{(p,p) \in \mathbb{H}^n \times \mathbb{H}^n\}$ is given by the set of points $\{(x,y) \in \mathbb{H}^n \times \mathbb{H}^n\}$ such that
$$
(x, y)=\left(\cosh \left(\tfrac{l}{\sqrt{2}}\right) p+\sqrt{2} \sinh \left(\tfrac{l}{\sqrt{2}}
\right) v, \cosh \left(\tfrac{l}{\sqrt{2}}\right) p-\sqrt{2} \sinh \left(\tfrac{l}{\sqrt{2}}
\right) v\right),
$$
where $p \in \mathbb{H}^{n}$, $v \in T_{p} \mathbb{H}^{n}$, and $\|v\| = \tfrac{1}{\sqrt{2}}$.
From the fact that $\langle x, y\rangle = -\cosh (\sqrt{2} l)$, it follows that the hypersurface $M_\tau$ is a tube of radius $\frac{1}{\sqrt{2}}{\rm arccosh} (-\tau)$ over $\{(p,p) \in \mathbb{H}^n \times \mathbb{H}^n\}$. It also means that the hypersurfaces $M_\tau$, $\tau<-1$ are mutually parallel tubes, and their focal submanifold is the submanifold $\{(p,p) \in \mathbb{H}^n \times \mathbb{H}^n\}$. 
Thus, hypersurfaces $\{M_\tau,\ \tau<-1\}$ are isoparametric hypersurfaces.
\end{example}
	
\begin{remark}\label{rem:3.1aa}
In Theorem 13 of \cite{DDO}, D\'{\i}az-Ramos, Dom\'{\i}nguez-V\'{a}zquez and Otero classified the homogeneous hypersurfaces in 
$\H^{m_1}\times \H^{m_2}\times \cdots \times \H^{m_k}$. There are four types of homogeneous hypersurfaces in $\mathbb{H}^m\times \mathbb{H}^n$, namely $({\rm FH})$, $({\rm FS})$, $({\rm CEI})$ and $({\rm CER})$. We point out that these homogeneous hypersurfaces are exactly the 
$\Sigma\times \mathbb{H}^n$, where $\Sigma$ is a hypersurface of $\mathbb{H}^m$ with constant principal curvatures (see Example \ref{exam:1}); or $\mathbb{H}^m\times \tilde{\Sigma}$, where $\tilde{\Sigma}$ is a hypersurface of $\mathbb{H}^n$ with constant principal curvatures (see Example \ref{exam:2}); or $M_{1,-1}^{C}$ for some $C\in(-1,1)$ (see Example \ref{exam:3}); or $M_{1,1}^{C}$ for some $C\in(-1,1)$ (see Example \ref{exam:4}); or $M_\tau$ when $m = n$ for some $\tau<-1$ (see Example \ref{exam:5}). 
\end{remark}

\section{Cartan's formulas}\label{sect:4} 
	
Let $M$ be a hypersurface of $\mathbb{H}^m\times \mathbb{H}^n$ with constant principal curvatures and constant product angle function $C\neq\pm1$. Then by Corollary \ref{coro:2.2}, 
$V$ is a principal vector field of $M$, and it satisfies $AV=0$ and $\nabla_V V = 0$. Moreover, $V^\perp$ is $A$-invariant. 
Suppose tangent vector fields $X,Y\in TM$ satisfy $AX=\lambda X$ and $AY=\mu Y$. Then
\begin{equation}\label{eqn:4.1}
\langle(\nabla_Z A)X,Y\rangle=(\lambda-\mu)\langle\nabla_Z X,Y\rangle,  
\end{equation}
for any tangent vector $Z$.
	
Observe that the distribution $V^\perp$ is $A$-invariant. We denote the set of eigenvalues of $A$ restricted to $V^\perp$ by $\sigma(A|_{V^\perp})$. 
For any $\lambda\in\sigma(A|_{V^\perp})$, let $V_\lambda$ be the corresponding eigenspace restricted to $V^\perp$. 
	
\begin{lemma}\label{lemma:4.1}
Let $M$ be a hypersurface of $\mathbb{H}^m\times \mathbb{H}^n$ ($m\geq3, n\geq2$) with constant principal curvatures and constant product angle function $C\neq\pm1$.
\begin{enumerate}
\item For any $\lambda\in\sigma(A|_{V^\perp})$, there holds
$$
\left(\lambda^2-\frac{1}{2}(1-C^2)\right)\langle TX,Y\rangle=C\lambda^2\langle X,Y  \rangle, \ \ \forall \ X,Y\in V_\lambda. 
$$
			
\item For any $\lambda,\mu\in\sigma(A|_{V^\perp})$ and $\lambda\neq \mu$, there holds
$$
\left(\lambda\mu-\frac{1}{2}(1-C^2)\right)\langle TX,Y\rangle=(\lambda - \mu)\lang{\nabla_V X, Y}, \ \ \forall \ X\in V_\lambda,\ Y\in V_\mu. 
$$
\end{enumerate}
\end{lemma}
	
\begin{proof}
Suppose $X\in V_\lambda$ is a principal direction with respect to $\lambda \in \sigma(A|_{V^\perp})$. By Codazzi equation \eqref{eqn:2.3} and $AV=0$, we have 
\begin{equation}\label{eqn:lem:4.1}
(\nabla_{V}A)X - (\nabla_{X}A) V =\lambda \nabla_V X - A(\nabla_V X) + A(\nabla_X V) = -\frac{1}{2} (1-C^2) TX.
\end{equation}
		
(1) Taking the inner product of \eqref{eqn:lem:4.1} with $Y\in V_\lambda$, we have
$$
\lambda \lang{\nabla_X V, Y} = -\frac{1}{2} (1-C^2) \lang{TX,Y}.
$$
By Lemma \ref{lemma:2.1} and $AX=\lambda X$, we can get
$$
\left(\lambda^2-\frac{1}{2}(1-C^2)\right)\langle TX,Y\rangle=C\lambda^2\langle X,Y  \rangle, \ \ \forall \ X,Y\in V_\lambda. 
$$
		
(2) Taking the inner product of \eqref{eqn:lem:4.1} with $Y\in V_\mu$, with the use of Lemma \ref{lemma:2.1}, we have
$$
\left(\lambda\mu-\frac{1}{2}(1-C^2)\right)\langle TX,Y\rangle=(\lambda - \mu)\lang{\nabla_V X, Y}, \ \ \forall \ X\in V_\lambda,\ Y\in V_\mu.
$$
\end{proof}
	
\begin{remark}\label{rem:denominater}
Let $Y=X$ and $\|X\|=1$ in Lemma \ref{lemma:4.1} (1), we have
$$
\left(\lambda^2-\frac{1}{2}(1-C^2)\right)\langle TX,X\rangle=C\lambda^2.
$$
We point out that $\lambda^2-\frac{1}{2}(1-C^2)=0$ if and only if $C = 0$ and $\lambda^2=\frac{1}{2}$. 
In fact, by $\lambda^2-\frac{1}{2}(1-C^2)=0$, 
we have $C\lambda^2 = 0$, hence $C=0$ or $\lambda=0$. If $\lambda = 0$, then $C^2 = 1$, which contradicts the assumption that $C\neq\pm 1$. If $C = 0$, then $\lambda^2 = \frac{1}{2}$. 
\end{remark}
	
\begin{corollary}\label{coro:4.2}
Let $M$ be a hypersurface of $\mathbb{H}^m\times \mathbb{H}^n$ ($m\geq3, n\geq2$) with constant principal curvatures and constant product angle function $C \neq \pm 1$.
\begin{enumerate}
\item[(1)]
If $\lambda=0$, then $TV_0\perp V_0$. 
\item[(2)]
If $C=0$ and $\lambda^2\neq\frac{1}{2}$, then  $TV_\lambda\perp V_\lambda$. 
\end{enumerate}
\end{corollary}
	
\begin{proof}
When $\lambda = 0$, by $C\neq \pm 1$, it holds that $\lambda^2-\frac{1}{2}(1-C^2)\neq 0$. Then by Lemma \ref{lemma:4.1} (1), we have
$$
\langle TX,Y\rangle = 0,\ \ \forall \ X,Y\in V_0, 
$$
which implies that $TV_0 \perp V_0$. 
		
When $C=0$ and $\lambda^2\neq\frac{1}{2}$, then $\lambda^2-\frac{1}{2}(1-C^2)\neq 0$. By Lemma \ref{lemma:4.1} (1), we also have 
$$
\langle TX,Y\rangle = 0,\ \ \forall \ X,Y\in V_\lambda,
$$
which implies $TV_\lambda \perp V_\lambda$.
\end{proof}

\begin{lemma}\label{lemma:4.3}
Let $M$ be a hypersurface of $\mathbb{H}^m\times \mathbb{H}^n$ ($m\geq3, n\geq2$) with constant principal curvatures and constant product angle function $C\neq\pm1$. For all $\lambda,\mu$ in $\sigma(A|_{V^{\perp}})$, we have
\begin{enumerate}
\item[(1)] $\nabla_X Y\perp V_\lambda$ if $X\in V_\lambda$, $Y\in V_\mu$, $\lambda\neq\mu$; 
\item[(2)] $\nabla_X Y+\frac{\lambda\big(C\langle X,Y\rangle-\langle TX,Y\rangle\big)}{1-C^2}V\in V_\lambda$ if $X,Y\in V_\lambda$. 
\end{enumerate}
\end{lemma}
	
\begin{proof}
(1) Take $X,Z\in V_\lambda$, $Y\in V_\mu$ and $\lambda\neq \mu$, by Codazzi equation \eqref{eqn:2.3}, we have
$$
0 = \langle(\nabla_{X} A)Y - (\nabla_{Y} A)X, Z\rangle = (\mu - \lambda)\langle\nabla_{X}Y, Z\rangle.
$$
Since $Z$ is arbitrary, it follows that $\nabla_X Y\perp V_\lambda$.
		
\vskip 4mm
		
(2) For $X,Y\in V_\lambda$ and $Z\in V_\mu$, it follows from (1) that  $\langle \nabla_XY,Z\rangle=-\langle \nabla_XZ,Y\rangle=0$, 
which implies that $\nabla_{X}Y\in V_\lambda \oplus {\rm Span}\{V\}$. By Lemma \ref{lemma:2.1}, we have $\nabla_X V = CAX - TAX$, hence 
$$
\lang{\nabla_{X}Y, V} = -\lang{Y, \nabla_X V} = -\lambda\lang{Y , CX - TX} = -\lambda(C\lang{X,Y} - \lang{TX,Y}).
$$
We can conclude
$$
\nabla_X Y+\frac{\lambda\big(C\langle X,Y\rangle-\langle TX, Y\rangle\big)}{1-C^2}V\in V_\lambda,\quad \forall X,Y\in V_\lambda.
$$
\end{proof}
	
In the following, we give Cartan's formulas (see \eqref{eqn:car-general},  \eqref{eqn:car-before}, \eqref{eqn:car-2}, \eqref{eqn:car-seperate-1} and \eqref{eqn:car-seperate-2}) for the hypersurface of $\mathbb{H}^m \times \mathbb{H}^n$ with constant principal curvatures and constant product angle function $C \neq \pm 1$. 
	
\begin{lemma}\label{lemma:car-general}
Let $M$ be a hypersurface of $\mathbb{H}^m\times \mathbb{H}^n$ ($m\geq3, n\geq2$) with constant principal curvatures and constant product angle function $C\neq\pm1$. 
Let $X\in V^{\perp}$ be a unit principal vector at a point $p$ with associated principal curvature $\lambda$. 
For any principal orthonormal basis $\{e_i\}_{i=1}^{m+n-2}$ of $V^{\perp}$ satisfying $Ae_i=\mu_ie_i$, we have 
\begin{equation}\label{eqn:car-general}
\begin{aligned}
\sum_{i=1, \mu_i\neq\lambda}^{m+n-2}\frac{1}{\lambda-\mu_i}
\Big\{&\lambda\mu_i\big(1+\frac{(C-\langle TX,X\rangle)(C-\langle Te_i,e_i\rangle)-2\langle TX,e_i\rangle^2}{1-C^2}\big)\\
&\quad\quad -\frac{1}{2}\big(1+\langle TX,X\rangle\langle Te_i,e_i\rangle-2\langle TX,e_i\rangle^2\big)\Big\}=0.
\end{aligned}
\end{equation}
In particular, when $C=0$, we have
\begin{equation}\label{eqn:car-0-general}
\sum_{i=1, \mu_i\neq\lambda}^{m+n-2}\frac{\lambda\mu_i-\frac{1}{2}}{\lambda-\mu_i} \Big(1+\langle TX,X\rangle\langle Te_i,e_i\rangle-2\langle TX,e_i\rangle^2\Big)=0.
\end{equation}
\end{lemma}
	
\begin{proof}
Extend $X$ and $\{e_i\}_{i=1}^{m+n-2}$ to be principal vector fields near $p$. For any fixed $e_i$ satisfying $Ae_i=\mu_i e_i$ and $\mu_i\neq \lambda$, the definition of Riemannian curvature gives
\begin{equation}\label{term-1}
R(X,e_i,e_i,X) = \underbrace{\langle\nabla_{X}\nabla_{e_i} e_i,X\rangle}_{(I)}\underbrace{ - \langle\nabla_{e_i}\nabla_{X} e_i,X\rangle}_{(II)}\underbrace{- \langle\nabla_{[X,e_i]} e_i,X\rangle}_{(III)}.
\end{equation}
In what follows, we will compute these three terms separately. 
		
\vskip 4mm
		
By Lemma \ref{lemma:4.3}, we have
$$
\nabla_{e_i} e_{i}+\frac{\mu_i\big(C-\langle Te_{i},e_{i}\rangle\big)}{1-C^2}V\in V_{\mu_i},\quad \nabla_X X+\frac{\lambda\big(C-\langle TX,X\rangle\big)}{1-C^2}V\in V_\lambda.
$$
So the first term $(I)$ is given by
\begin{equation}\label{term-1-I}
\begin{aligned}
(I) = &\nabla_{X}\langle \nabla_{e_i} e_i,X\rangle - \langle\nabla_{e_i} e_i,\nabla_{X}X\rangle = - \langle\nabla_{e_i} e_i,\nabla_{X}X\rangle \\
= & -\frac{\lambda\mu_{i}}{1-C^2}(C - \langle Te_i, e_i\rangle)(C - \langle TX, X\rangle).
\end{aligned}
\end{equation}
		
By Lemma \ref{lemma:4.3}, we have  $\nabla_{X}e_i\perp V_{\lambda}$. Therefore 
\begin{equation}\label{term-1-II}
(II) = -\nabla_{e_i}\langle\nabla_{X}e_i,X \rangle + \langle \nabla_{X}e_i,\nabla_{e_i}X \rangle = \langle \nabla_{X}e_i,\nabla_{e_i}X \rangle.
\end{equation}
		
By using \eqref{eqn:4.1} and Codazzi equation \eqref{eqn:2.3}, the third term $(III)$ can be expressed as
\begin{equation}\label{term-1-III}
\begin{aligned}
(III)= & - \langle\nabla_{[X,e_i]} e_i,X\rangle = \frac{1}{\lambda - \mu_{i}}\langle (\nabla_{[X, e_i]} A)e_i, X \rangle\\
= & \frac{1}{\lambda - \mu_{i}}\Big(\langle  (\nabla_{e_i}A)[X,e_i],X\rangle - \frac{1}{2}\langle [X, e_i], V\rangle \langle Te_i, X \rangle\Big) \\
= &  \frac{1}{\lambda - \mu_{i}}\langle  [X,e_i] , (\nabla_{e_i}A)X \rangle - \frac{1}{2}\langle TX, e_i\rangle^2\\
= & \frac{1}{\lambda - \mu_{i}}\Big(\langle \nabla_{X}e_i , (\nabla_{e_i}A)X  \rangle - \langle \nabla_{e_i}X , (\nabla_{e_i}A)X \rangle\Big) - \frac{1}{2}\langle TX, e_i\rangle^2 \\
= & \frac{1}{\lambda - \mu_{i}}\Big(\langle \nabla_{X}e_i , (\nabla_{e_i}A)X \rangle - \langle \nabla_{e_i}X , (\nabla_{X}A)e_i \rangle\Big) - \frac{1}{2}\langle TX, e_i\rangle^2 \\
= & \frac{1}{\lambda - \mu_{i}}\Big( \lambda\langle \nabla_{X}e_i , \nabla_{e_i}X \rangle - \langle \nabla_{X}e_i , A(\nabla_{e_i}X) \rangle \\
& - \mu_i\langle \nabla_{e_i}X , \nabla_{X}e_i \rangle + \langle \nabla_{e_i}X , A(\nabla_{X}e_i) \rangle \Big) - \frac{1}{2}\langle TX, e_i\rangle^2 \\
= & \langle \nabla_{X}e_i , \nabla_{e_i}X \rangle - \frac{1}{2}\langle TX, e_i\rangle^2.
\end{aligned}
\end{equation}
		
\vskip 4mm
		
Substituting \eqref{term-1-I}, \eqref{term-1-II} and \eqref{term-1-III} into \eqref{term-1}, we obtain
$$
R(X,e_i,e_i,X) = 2\langle \nabla_{X}e_i , \nabla_{e_i}X \rangle - \frac{1}{2}\langle TX, e_i\rangle^2 - \frac{\lambda\mu_{i}}{1-C^2}\big(C - \langle TX, X\rangle\big)\big(C - \langle Te_i, e_i\rangle\big). 
$$
On the other hand, Gauss equation \eqref{eqn:2.2} gives
$$
R(X,e_i,e_i,X) = -\frac{1}{2}\Big(1 + \lang{TX, X}\lang{Te_i,e_i} - \lang{TX,e_i}^2\Big)+ \lambda\mu_i.
$$
Therefore, we get
\begin{equation}\label{eqn:4.1sa1}
\begin{aligned}
&2\langle \nabla_{X}e_i , \nabla_{e_i}X \rangle - \frac{1}{2}\langle TX, e_i\rangle^2 - \frac{\lambda\mu_{i}}{1-C^2}\big(C - \langle TX, X\rangle\big)\big(C - \langle Te_i, e_i\rangle \big)  \\
= & -\frac{1}{2}\big(1 + \langle TX, X\rangle\langle Te_i, e_i\rangle - \langle TX, e_i\rangle^2\big) + \lambda\mu_i.
\end{aligned}
\end{equation}
In the following, we express $\langle \nabla_{X}e_i , \nabla_{e_i}X \rangle$ in terms of the orthonormal principal basis $\{e_k, \bar{V}\}_{k=1}^{m+n-2}$, where $\bar V = \frac{1}{\sqrt{1-C^2}}V$. By Lemma \ref{lemma:2.1}, 
we have
\begin{equation}\label{eqn:4.1sa2}
\begin{aligned}
2\langle \nabla_{X}e_i , \nabla_{e_i}X \rangle = & 2\sum_{k=1}^{m+n-2} \langle \nabla_{X}e_i , e_{k} \rangle\langle \nabla_{e_i}X , e_{k} \rangle + 2\langle \nabla_{X}e_i , \bar V \rangle\langle \nabla_{e_i}X , \bar V \rangle \\
= & 2\sum_{k;\mu_{k}\neq \mu_{i},\mu_{k}\neq \lambda} \frac{\langle (\nabla_{X}A)e_i , e_{k} \rangle\langle (\nabla_{e_i}A)X , e_{k} \rangle}{(\mu_i - \mu_{k})(\lambda - \mu_{k})} + \frac{2\lambda\mu_i\langle TX, e_i\rangle^2}{1-C^2},
\end{aligned}
\end{equation}
where we have used that $\nabla_X e_i\perp V_\lambda$ and $\nabla_{e_i} X \perp V_{\mu_i}$.
		
\vskip 4mm
		
Now, using \eqref{eqn:4.1sa1}, \eqref{eqn:4.1sa2} and Codazzi equation \eqref{eqn:2.3}, we obtain
\begin{equation}\label{eqn:4.1sa3}
\begin{aligned}
& 2\sum_{k;\mu_{k}\neq \mu_{i},\mu_{k}\neq \lambda} \frac{\langle (\nabla_{X}A)e_i , e_{k} \rangle^2}{(\mu_i - \mu_{k})(\lambda - \mu_{k})}
= 2 \sum_{k;\mu_{k}\neq \mu_{i},\mu_{k}\neq \lambda} \frac{\langle (\nabla_{X}A)e_i , e_{k} \rangle\langle (\nabla_{e_i}A)X , e_{k} \rangle}{(\mu_i - \mu_{k})(\lambda - \mu_{k})} \\
= &  \left( \lambda\mu_i - \frac{1}{2} - \frac{1}{2} \langle TX, X\rangle\langle Te_i, e_i\rangle + \frac{\lambda\mu_i}{1-C^2}(C - \langle TX, X\rangle)(C - \langle Te_i, e_i\rangle) \right)  \\
&\quad\quad - \frac{2\lambda\mu_i - (1-C^2)}{1-C^2}\langle TX , e_i\rangle^2.
\end{aligned}
\end{equation}
Dividing \eqref{eqn:4.1sa3} by $(\lambda - \mu_i)$ and summing over $i$, we have
\begin{equation}\label{eqn:4.1sa4}
\begin{aligned}
2\sum_{i;\mu_i\neq \lambda}&\Big(\sum_{k;\mu_{k}\neq \mu_{i},\mu_{k}\neq \lambda} \frac{\langle (\nabla_{X}A)e_i , e_{k} \rangle^2}{(\lambda - \mu_i)(\mu_i - \mu_{k})(\lambda - \mu_{k})}\Big)\\
= & \sum_{i=1, \mu_i\neq\lambda}^{m+n-2}\frac{1}{\lambda-\mu_i}
\Big\{\lambda\mu_i\big(1+\frac{(C-\langle TX,X\rangle)(C-\langle Te_i , e_i\rangle)-2\langle TX,e_i\rangle^2}{1-C^2}\big)\\
&\quad \quad -\frac{1}{2}\big(1+\langle TX,X\rangle\langle Te_i, e_i\rangle-2\langle TX,e_i\rangle^2\big)\Big\}.
\end{aligned}
\end{equation}
Since the summand on the left side of \eqref{eqn:4.1sa4} is skew-symmetric in $\{i,k\}$, 
the value of the sum is $0$, and so the sum on the right is $0$. 
\end{proof}
	
\begin{remark}\label{rem:4.1a}
The proof of Lemma \ref{lemma:car-general} is inspired by Cartan's formulas for
Hopf hypersurfaces of $\mathbb{C}H^m$ with constant principal curvatures, which 
are obtained by Berndt \cite{B}. For more details of Cartan's formulas, one can also see \cite{CR15, LTY25,LTY26}.
\end{remark}
	
Now, from \eqref{eqn:4.1sa3}, we have the following key lemma:
	
\begin{lemma}\label{lemma:car-before}
Let $M$ be a hypersurface of $\mathbb{H}^m \times \mathbb{H}^n$ ($m\geq3, n\geq2$) with constant principal curvatures and constant product angle function $C \neq \pm 1$. Let $X, Y\in V^{\perp}$ be unit principal vector fields with associated principal curvatures $\lambda , \mu$ ($\mu \neq \lambda$). For any principal orthonormal basis $\{e_i\}_{i=1}^{m+n-2}$ of $V^{\perp}$ satisfying $Ae_i=\mu_ie_i$, we have
\begin{equation}\label{eqn:car-before}
\begin{aligned}
2\sum_{\mu_i\neq\lambda,\mu}\frac{\lang{(\nabla_{e_i} A)X,Y}^2}{(\lambda-\mu_i)(\mu-\mu_i)}
=&\lambda\mu\Big(1+\frac{(C-\langle TX,X\rangle)(C-\langle TY,Y\rangle)-2\langle TX,Y\rangle^2}{1-C^2}\Big)\\
&\quad \quad -\frac{1}{2}\big(1+\langle TX,X\rangle\langle TY,Y\rangle-2\langle TX,Y\rangle^2\big). 
\end{aligned}
\end{equation}
In particular, when $C=0$, we have
\begin{equation}\label{eqn:car-0-before}
2\sum_{\mu_i\neq\lambda,\mu}\frac{\lang{(\nabla_{e_i} A)X,Y}^2}{(\lambda-\mu_i)(\mu-\mu_i)} = \br{\lambda\mu - \frac{1}{2}}\Big(1+\langle TX,X\rangle\langle TY,Y\rangle-2\langle TX,Y\rangle^2\Big).
\end{equation}
\end{lemma}
	
If the number of distinct principal curvatures of $\sigma(A|_{V^\perp})$ is $2$, then we have the following result. 
	
\begin{lemma}\label{lemma:car-2}
Let $M$ be a hypersurface of $\mathbb{H}^m\times \mathbb{H}^n$ ($m\geq3, n\geq2$) with constant principal curvatures and constant product angle function $C\neq\pm1$. 
Assume that the number of distinct principal curvatures of $\sigma(A|_{V^\perp})$ is $2$, i.e. $\sigma(A|_{V^\perp})=\{\lambda,\mu\}$, $\lambda\neq \mu$. Then for any $X\in V^{\perp}$ and $Y\in V^{\perp}$ being unit principal vector fields with associated principal curvatures $\lambda$ and $\mu$ respectively, we have 
\begin{equation}\label{eqn:car-2}
\begin{aligned}
\lambda\mu\Big(1+\tfrac{(C-\langle TX,X\rangle)(C-\langle TY,Y\rangle)-2\langle TX,Y\rangle^2}{1-C^2}\Big)-\frac{1}{2}\big(1+\langle TX,X\rangle\langle TY,Y\rangle-2\langle TX,Y\rangle^2\big)=0. 
\end{aligned}
\end{equation}
In particular, when $C=0$, we have
\begin{equation}\label{eqn:car-0-2}
\br{\lambda\mu - \frac{1}{2}}\Big(1+\langle TX,X\rangle\langle TY,Y\rangle-2\langle TX,Y\rangle^2\Big) = 0.
\end{equation}
\end{lemma}
\begin{proof}
According to the assumption that the number of distinct principal curvatures of $\sigma(A|_{V^\perp})$ is $2$, then following the
proof of Lemma \ref{lemma:car-general}, we have 
$$
\langle\nabla_X Y,\nabla_Y X\rangle=\frac{\lambda\mu\langle TX, Y\rangle^2}{1-C^2}, 
$$
and
$$
\begin{aligned}
&2\langle \nabla_{X}Y, \nabla_{Y}X \rangle - \frac{1}{2}\langle TX, Y\rangle^2 - \frac{\lambda\mu}{1-C^2}\big(C - \langle TX, X\rangle\big)\big(C - \langle TY, Y\rangle \big)  \\
= & -\frac{1}{2}\big(1 + \langle TX, X\rangle\langle TY, Y\rangle - \langle TX, Y\rangle^2\big) + \lambda\mu.
\end{aligned}
$$
Combining these two equations, we obtain \eqref{eqn:car-2}. 
\end{proof}
	
\begin{lemma}\label{lemma:car-seperate}
Let $M$ be a hypersurface of $\mathbb{H}^m\times \mathbb{H}^n$ ($m\geq3, n\geq2$) with constant principal curvatures and constant product angle function $C\neq\pm1$. Let $X\in V^{\perp}$ be a unit principal vector at a point $p$ with associated principal curvature $\lambda$. If $\lambda\neq 0$ and $C\neq 0$ (resp. $C = 0$, $\lambda\neq 0$ and $\lambda^2 \neq \frac{1}{2}$), then for any principal orthonormal basis $\{e_i\}_{i=1}^{m+n-2}$ of $V^{\perp}$ satisfying $Ae_i=\mu_ie_i$, we have 
\begin{equation}\label{eqn:car-seperate-1}
\begin{aligned}
\sum_{i=1, \mu_i\neq\lambda}^{m+n-2}\frac{1}{\lambda-\mu_i} \Big \{ & \lambda\mu_i\big(1+\frac{(C-\langle TX,X\rangle) (C-\langle Te_i,e_i\rangle)}{1-C^2}\big)-\frac{1}{2} \big( 1 + \langle TX,X\rangle\langle Te_i, e_i\rangle\big)\Big\}=0, 
\end{aligned}
\end{equation}
\begin{equation}\label{eqn:car-seperate-2}
\begin{aligned}
\sum_{i=1, \mu_i\neq\lambda}^{m+n-2}\frac{\lambda\mu_i-\frac{1}{2}(1-C^2)}{\lambda-\mu_i}\langle TX,e_i\rangle^2=0.
\end{aligned}
\end{equation}
In particular, when $C = 0$, $\lambda\neq 0$, and $\lambda^2\neq \frac{1}{2}$, we have
\begin{equation}\label{eqn:car-0-seperate-1}
\sum_{i=1, \mu_i\neq\lambda}^{m+n-2}\frac{\lambda\mu_i-\frac{1}{2}}{\lambda-\mu_i}
=0, 
\end{equation}
		
\begin{equation}\label{eqn:car-0-seperate-2}
\sum_{i=1, \mu_i\neq\lambda}^{m+n-2}\frac{\lambda\mu_i-\frac{1}{2}}{\lambda-\mu_i}\langle TX,e_i\rangle^2=0.
\end{equation}
\end{lemma}

\begin{proof}
Consider the sectional curvature $R(X, \bar V, \bar V, X)$, where $\bar V = \frac{V}{\sqrt{1-C^2}}$. The definition of Riemannian curvature gives
$$
(1-C^2)R(X, \bar V , \bar V , X) = \underbrace{\langle\nabla_{X}\nabla_{V} V,X\rangle}_{(I)}\underbrace{ - \langle\nabla_{V}\nabla_{X} V,X\rangle}_{(II)}\underbrace{- \langle\nabla_{[X,V]} V,X\rangle}_{(III)}.
$$
By Corollary \ref{coro:2.2}, we have $\nabla_V V = 0$, hence $(I)= 0$. 
In the following, we compute the terms (II) and (III). 
		
Recall that in Lemma \ref{lemma:4.1}, we have 
\begin{equation}\label{eqn:4.1ff1}
\left(\lambda^2-\frac{1}{2}(1-C^2)\right)\langle TX,Y\rangle=C\lambda^2\langle X,Y  \rangle, \ \ \forall \ X,Y\in V_\lambda. 
\end{equation}
As we discussed in Remark \ref{rem:denominater}, when $C\neq 0$ (or $C = 0$ and $\lambda^2\neq \frac{1}{2}$), it holds that $\lambda^2-\frac{1}{2}(1-C^2)\neq 0$. Then by \eqref{eqn:4.1ff1}, we have  
$$
\langle\nabla_{X} V,X\rangle = \lambda\lang{CX - TX, X} = \lambda\left(C - \frac{C\lambda^2}{\lambda^2-\frac{1}{2}(1-C^2)} \right),
$$
which means that $\langle\nabla_{X} V,X\rangle$ is a constant. As a consequence, the second term $(II)$ is
\begin{equation}\label{eqn:4.1ffs1}
(II)= - V(\langle\nabla_{X} V,X\rangle)+ \langle\nabla_{X} V,\nabla_{V} X\rangle = \langle\nabla_{X} V,\nabla_{V} X\rangle.
\end{equation}
From Corollary \ref{coro:2.2}, $\nabla_V V = 0$, we have 
$\langle [X, V], V \rangle = \lang{\nabla_X V, V} - \lang{\nabla_V X, V} = 0$. Then by \eqref{eqn:4.1} and Codazzi equation \eqref{eqn:2.3}, the third term $(III)$ can be expressed as
\begin{equation}\label{eqn:4.1ff2}
\begin{aligned}
(III)=& - \langle\nabla_{[X,V]} V,X\rangle = \frac{1}{\lambda}\langle (\nabla_{[X,V]}A) V,X \rangle\\
= & \frac{1}{\lambda}\left(\langle ( \nabla_{V}A) [X,V],X \rangle + \frac{1}{2}(1-C^2)\left\langle T[X,V],X \right\rangle  \right) \\
= & \frac{1}{\lambda}\left(\left\langle [X,V], (\nabla_{V}A) X \right\rangle + \frac{1}{2}(1-C^2)\langle [X,V], TX\rangle\right).
\end{aligned}
\end{equation}
On the other hand, by using Codazzi equation \eqref{eqn:2.3}, we have
\begin{equation}\label{eqn:4.1ff3}
\begin{aligned}
& \left\langle [X,V], (\nabla_{V}A) X \right\rangle = \langle \nabla_{X} V, (\nabla_{V}A) X \rangle - \langle \nabla_{V}X, (\nabla_{V}A) X \rangle \\
= & \langle \nabla_{X}V, (\nabla_{V}A) X \rangle - \langle \nabla_{V}X, (\nabla_{X}A) V \rangle + \frac{1}{2}(1-C^2)\langle \nabla_{V}X, TX \rangle \\
= & \lambda\langle \nabla_{X}V, \nabla_{V}X \rangle - \langle \nabla_{X}V, A(\nabla_{V}X) \rangle+ \langle \nabla_{V}X, A(\nabla_{X} V) \rangle+ \frac{1}{2}(1-C^2)\langle \nabla_{V}X, TX \rangle\\
= & \lambda \langle \nabla_{X} V, \nabla_{V}X \rangle + \frac{1}{2}(1-C^2)\langle \nabla_{V}X,TX \rangle.
\end{aligned}
\end{equation}
Hence by \eqref{eqn:4.1ff2} and \eqref{eqn:4.1ff3}, we get 
\begin{equation}\label{eqn:4.1ff4}
\begin{aligned}
(III)= & \frac{1}{\lambda}\left(\left\langle [X,V], (\nabla_{V}A) X \right\rangle + \frac{1}{2}(1-C^2)\langle [X,V], TX\rangle\right)\\
= & \frac{1}{\lambda}\br{\lambda \langle \nabla_{X} V, \nabla_{V}X \rangle + \frac{1}{2}(1-C^2)\langle \nabla_{V}X,TX \rangle + \frac{1}{2}(1-C^2)\langle [X,V], TX\rangle}\\
= & \langle \nabla_{X} V, \nabla_{V}X \rangle + \frac{1}{2\lambda}(1-C^2)\langle \nabla_{X}V,TX \rangle\\
= & \langle \nabla_{X}V, \nabla_{V}X \rangle + \frac{1}{2}(1-C^2)(C\langle TX, X\rangle - 1).
\end{aligned}
\end{equation}
		
\vskip 4mm
		
Combining $(I)=0$, \eqref{eqn:4.1ffs1} and \eqref{eqn:4.1ff4}, we obtain
$$
(1-C^2)R(X, \bar V , \bar V , X) = 2\langle \nabla_{X}V, \nabla_{V}X \rangle + \frac{1}{2}(1-C^2)(C\langle TX, X\rangle - 1).
$$
On the other hand, Gauss equation implies that 
$$
R(X, \bar V, \bar V, X)  = -\frac{1}{2}\left( 1 - C \langle TX, X\rangle \right).
$$
Thus, the preceding two equations yield
\begin{equation}\label{eqn:4.1ff5}
\langle \nabla_{V}X, \nabla_{X}V\rangle = 0.
\end{equation}
By Lemma \ref{lemma:2.1}, $\nabla_X V = CAX - TAX$ and \eqref{eqn:4.1ff5}, we get 
$$
0 = \langle \nabla_{V}X, \nabla_{X}V\rangle = \lambda\langle \nabla_{V} X, CX - TX\rangle = -\lambda\langle \nabla_{V} X,  TX\rangle.
$$
It follows from $\lambda\neq0$ that $\langle \nabla_{V} X,  TX\rangle=0$. By \eqref{eqn:4.1ff1} and $\lambda^2-\frac{1}{2}(1-C^2)\neq 0$, we know that 
$$
\langle TX, Y\rangle=0, \ \text{for any $X,Y\in V_\lambda$ and $X\perp Y$}.
$$
Hence, for any unit vector $X\in V_\lambda$, we have 
$$
\begin{aligned}
0 = & \langle \nabla_{V}X, TX\rangle = \sum_{i = 1,\mu_i\neq \lambda}^{m+n-2}\langle \nabla_{V}X, e_i\rangle\langle TX, e_i\rangle + \langle \nabla_{V}X, X\rangle\langle TX, X\rangle + \langle \nabla_{V}X, \bar V\rangle\langle TX, \bar V\rangle \\
= & \sum_{i = 1,\mu_i\neq \lambda}^{m+n-2}\langle \nabla_{V}X, e_i\rangle\langle TX, e_i\rangle = \sum_{i = 1,\mu_i\neq \lambda}^{m+n-2} \frac{\lambda\mu_i - \frac{1}{2}(1-C^2)}{\lambda - \mu_i}\langle TX, e_i\rangle^2,
\end{aligned}
$$
where we have used Lemma \ref{lemma:4.1} (2). Thus we get \eqref{eqn:car-seperate-2}. Now, combining \eqref{eqn:car-seperate-2} and \eqref{eqn:car-general}, we can get \eqref{eqn:car-seperate-1}. 
		
In particular, when $C = 0$, $\lambda\neq 0$ and $\lambda^2\neq \frac{1}{2}$, \eqref{eqn:car-seperate-2} and \eqref{eqn:car-seperate-1} reduce to 
\begin{equation}\label{eqn:4.s3}
\sum_{i=1, \mu_i\neq\lambda}^{m+n-2}\frac{\lambda\mu_i-\frac{1}{2}}{\lambda-\mu_i}\br{1+\langle TX,X\rangle\langle Te_i,e_i\rangle}= 0,\quad \sum_{i=1, \mu_i\neq\lambda}^{m+n-2}\frac{\lambda\mu_i-\frac{1}{2}}{\lambda-\mu_i}\langle TX,e_i\rangle^2=0. 
\end{equation}
Noticing that when $C = 0$, $\lambda\neq 0$ and $\lambda^2\neq \frac{1}{2}$, by Corollary \ref{coro:4.2}, it holds that $TV_\lambda \perp V_\lambda$, which means that $\lang{TX, X} = 0$ for $X\in V_\lambda$. Therefore, \eqref{eqn:4.s3} can further reduce to 
\eqref{eqn:car-0-seperate-1} and \eqref{eqn:car-0-seperate-2}. 
\end{proof}

\section{Hypersurfaces of $\mathbb{H}^m\times \mathbb{H}^n$ with at most two distinct constant principal curvatures}\label{sect:5}
	
In this section, we study the hypersurfaces in $\mathbb{H}^m\times\mathbb{H}^n$ ($m\geq3, n\geq2$) with constant principal curvatures. Let $g$ be the number of distinct constant principal curvatures. By restricting the number of distinct principal curvatures to be at most two, we obtain the following result. 
	
\begin{theorem}\label{thm:1.2}
Let $M$ be a connected oriented hypersurface of $\mathbb{H}^m\times \mathbb{H}^n$ ($m\geq3, n\geq2$) with at most two distinct constant principal curvatures. Then, up to isometries of $\mathbb{H}^m\times \mathbb{H}^n$, one of the following three cases occurs: 
\begin{itemize}
\item[(1)] $M$ is an open part of $\Sigma\times \mathbb{H}^n$, where $\Sigma$ is a totally umbilical hypersurface of $\mathbb{H}^m$ (see Example \ref{exam:1}); or
\item[(2)] $M$ is an open part of $\mathbb{H}^m\times \tilde{\Sigma}$, where $\tilde{\Sigma}$ is a totally umbilical hypersurface of $\mathbb{H}^n$ (see Example \ref{exam:2}); or
\item[(3)] $M$ is an open part of $M_{1,-1}^{0}$ (see Example \ref{exam:3}). 
\end{itemize}
\end{theorem}
\begin{proof}
Let $M$ be a hypersurface of $\mathbb{H}^m\times \mathbb{H}^n$ with at most two distinct constant principal curvatures. If $g=1$, then $M$ is a totally  
umbilical hypersurface. According to \cite{Toj13}, we know 
that $M$ is either an open part of $\mathbb{H}^{m-1}\times \mathbb{H}^n$ or an open part of $\mathbb{H}^m\times \mathbb{H}^{n-1}$. In the following, we only need to consider $g=2$. 
We assume that the principal curvatures are $\lambda$ and $\mu$. 
Let $\tilde{V}_\lambda$ and $\tilde{V}_\mu$ be the principal curvature eigenspaces on $M$, respectively. Then ${\rm dim} \tilde{V}_\lambda+{\rm dim} \tilde{V}_\mu=m+n-1$. 
Note that, for any $X,Y,Z\in \tilde{V}_\lambda$ (or $X,Y,Z\in \tilde{V}_\mu$), we always have 
\begin{equation}\label{eqn:5.1a1g1}
\langle (\nabla_{X} A)Y-(\nabla_Y A)X,Z \rangle=0.
\end{equation}
	
Let $M=M_1\cup M_2$, where 
$$
M_1=\{x\in M| C=\pm1\ \text{at point $x$}\}, \ \ 
M_2=\{x\in M| C\neq\pm1\ \text{at point $x$}\}. 
$$
If $M_2$ is nonempty, then $M_2$ is an open subset of $M$, and 
the vector field $V$ is nowhere vanishing on $M_2$. Furthermore, let 
$M_2=M_{21}\cup M_{22}$, where 
$$
\begin{aligned}
&M_{21}=\{x\in M_2| \text{ $V$ is a principal direction at point $x$}\}, \\
&M_{22}=\{x\in M_2| \text{ $V$ has two nontrivial projections on principal curvature eigenspaces at point $x$}\}. 
\end{aligned}
$$
	
We choose any given connected open subset $\Omega$ of $M_{22}$, 
and we can assume that 
$$
\bar{V}=\frac{V}{\sqrt{1-C^2}}=\cos(\theta) X_1+\sin(\theta) Y_1 
$$
on $\Omega$, where $\theta$ is a function on $\Omega$ such that $\cos(\theta)\sin(\theta)\neq0$, $X_1\in \tilde{V}_{\lambda}$ 
and $Y_1\in \tilde{V}_{\mu}$ are unit vector fields, respectively. 
Let $\bar{X}=\sin(\theta) X_1-\cos(\theta) Y_1$, then $\bar{X}\in V^\perp$ is a unit vector field, and we have 
\begin{equation}\label{eqn:5.1a1s1}
A\bar{V}=\cos(\theta) \lambda X_1+\sin(\theta) \mu Y_1,\ \ 
A\bar{X}=\sin(\theta) \lambda X_1-\cos(\theta) \mu Y_1. 
\end{equation}
	
If ${\rm dim}\tilde{V}_\lambda=1$ and ${\rm dim}\tilde{V}_\mu=m+n-2$, then for any vector fields $Y,Z\in \tilde{V}_\mu$ such that $Y,Z\perp Y_1$, using 
Codazzi equation \eqref{eqn:2.3}, we obtain 
\begin{equation}\label{eqn:5.1a1z1}
\begin{aligned}
\langle (\nabla_{Y_1} A)Y-(\nabla_Y A)Y_1,Z \rangle
&=-\frac{1}{2}(\langle Y_1,V \rangle\langle TY,Z \rangle-\langle Y,V \rangle\langle TY_1,Z \rangle)\\
&=-\frac{1}{2}\sqrt{1-C^2}\sin(\theta)\langle TY,Z \rangle.
\end{aligned}
\end{equation}
It follows from \eqref{eqn:5.1a1g1} that $T(\tilde{V}_\mu\ominus\mathbb{R}Y_1)\perp (\tilde{V}_\mu\ominus\mathbb{R}Y_1)$. Since $(\tilde{V}_\mu\ominus\mathbb{R}Y_1)\subset V^\perp$ and $TV^\perp=V^\perp$, we have 
$T(\tilde{V}_\mu\ominus\mathbb{R}Y_1)\subset (V^\perp\ominus(\tilde{V}_\mu\ominus\mathbb{R}Y_1))$, which contradicts ${\rm dim}(\tilde{V}_\mu\ominus\mathbb{R}Y_1)=m+n-3\geq2$ 
and ${\rm dim}(V^\perp\ominus(\tilde{V}_\mu\ominus\mathbb{R}Y_1))=1$. 
If ${\rm dim}\tilde{V}_\lambda=m+n-2$ and ${\rm dim}\tilde{V}_\mu=1$, we can also get a contradiction. Thus, we always have that ${\rm dim}\tilde{V}_\lambda\geq2$ and ${\rm dim}\tilde{V}_\mu\geq2$ on $\Omega$. 
	
For any vector field $Y\in \tilde{V}_\lambda$ such that $Y\perp X_1$, and any vector field $Z\in \tilde{V}_\lambda$, using Codazzi equation \eqref{eqn:2.3}, we obtain
\begin{equation}\label{eqn:5.1a1}
\begin{aligned}
\langle (\nabla_{X_1} A)Y-(\nabla_Y A)X_1,Z \rangle
&=-\frac{1}{2}(\langle X_1,V \rangle\langle TY,Z \rangle-\langle Y,V \rangle\langle TX_1,Z \rangle)\\
&=-\frac{1}{2}\sqrt{1-C^2}\cos(\theta)\langle TY,Z \rangle.
\end{aligned}
\end{equation}
It follows from \eqref{eqn:5.1a1g1}, $\cos(\theta)\sin(\theta)\neq0$ and $C\neq\pm1$ that 
\begin{equation}\label{eqn:5.1a1d1}
\langle TY,Z \rangle=0,\  \text{for any $Y,Z\in \tilde V_\lambda$ and $Y\perp X_1$.}
\end{equation}
Similarly, for any vector fields $Y,Z\in \tilde{V}_\mu$ and  $Y\perp Y_1$, by using Codazzi equation \eqref{eqn:2.3}, 
we have 
\begin{equation}\label{eqn:5.1a1d2}
\langle TY,Z \rangle=0,\ \text{for any $Y,Z\in \tilde{V}_\mu$ and $Y\perp Y_1$.}
\end{equation}
	
Let $p_1=\langle TX_1,X_1 \rangle$, $p_2=\langle TX_1,Y_1 \rangle$, 
$p_3=\langle TY_1,Y_1 \rangle$. Then we calculate that 
\begin{equation*}
\begin{aligned}
\langle TX_1,TY_1 \rangle&=\langle PX_1,PY_1 \rangle-\langle PX_1,N \rangle\langle PY_1,N \rangle=-\langle X_1,V \rangle\langle Y_1,V \rangle=-(1-C^2)\cos(\theta)\sin(\theta).
\end{aligned}
\end{equation*}
On the other hand, by \eqref{eqn:5.1a1d1} and \eqref{eqn:5.1a1d2}, we have 
$\langle TX_1,TY_1 \rangle=p_1p_2+p_2p_3=p_2(p_1+p_3).$
Since $C\neq\pm1$ and $\cos(\theta)\sin(\theta)\neq0$, 
we have  $p_2\neq0$ and $p_1+p_3\neq0$. 
	
Now, for any $Y\in \tilde V_\mu$ and $Y\perp Y_1$, since $Y,PY\in V^\perp$, 
we have 
\begin{equation*}
\begin{aligned}
\langle TX_1,TY \rangle&=\langle TX_1,PY \rangle=\langle PX_1,PY \rangle=0.
\end{aligned}
\end{equation*}
By \eqref{eqn:5.1a1d1} and \eqref{eqn:5.1a1d2}, we also have $\langle TX_1,TY \rangle=\langle TX_1,X_1 \rangle\langle X_1,TY \rangle=p_1\langle X_1,TY \rangle$ for any $Y\in \tilde{V}_\mu$ and $Y\perp Y_1$. Then, it holds  
\begin{equation}\label{eqn:5.1a4}
p_1\langle TX_1,Y \rangle=0,\ \text{for any $Y\in \tilde{V}_\mu$ and $Y\perp Y_1$.} 
\end{equation}
Similarly, for any $Z\in \tilde V_\lambda$ and $Z\perp X_1$, we can also have 
\begin{equation}\label{eqn:5.1a5}
p_3\langle TY_1,Z \rangle=0,\ \text{for any $Z\in \tilde{V}_\lambda$ and $Z\perp X_1$.}
\end{equation}
	
If at a point $x\in \Omega$, it holds $p_1p_3\neq0$. Then by \eqref{eqn:5.1a4} and 
\eqref{eqn:5.1a5}, we have $\langle TX_1,Y \rangle=0$ for any $Y\in \tilde{V}_\mu$ and $Y\perp Y_1$, $\langle TY_1,Z \rangle=0$ for any $Z\in \tilde V_\lambda$ and $Z\perp X_1$. It follows from \eqref{eqn:5.1a1d1} and \eqref{eqn:5.1a1d2} that $T({\rm Span}\{X_1,Y_1\})\subset {\rm Span}\{X_1,Y_1\}$. 
	
If $p_1=0$ and $p_3\neq0$ at $x\in \Omega$, then by \eqref{eqn:5.1a5}, we have 
\begin{equation}\label{eqn:5.1f1}
\langle TY_1,Z \rangle=0,\  \text{for any $Z\in \tilde{V}_\lambda$ and $Z\perp X_1$}.
\end{equation}
Then, for any $Z\in \tilde{V}_\mu$ and $Z\perp Y_1$, by $Z,PZ\in V^\perp$, we calculate that
\begin{equation}\label{eqn:5.1a5h1}
\begin{aligned}
\langle TY_1,TZ \rangle&=\langle TY_1,PZ \rangle=\langle PY_1,PZ \rangle=0.
\end{aligned}
\end{equation}
By \eqref{eqn:5.1f1}, we get  
$$
\langle TY_1,TZ \rangle=\langle TY_1,X_1 \rangle\langle X_1,TZ \rangle=p_2\langle TX_1,Z \rangle.
$$
It follows from \eqref{eqn:5.1a5h1} and $p_2\neq0$ that 
$\langle TX_1,Z \rangle=0$ for any $Z\in \tilde{V}_\mu$ and $Z\perp Y_1$.
Thus, we still have $T({\rm Span}\{X_1,Y_1\})\subset {\rm Span}\{X_1,Y_1\}$. 
	
If at point $x\in \Omega$, it holds $p_1\neq0$ and $p_3=0$. Then by \eqref{eqn:5.1a4}, we have 
\begin{equation}\label{eqn:5.1f2}
\langle TX_1,Y \rangle=0,\ \text{for any $Y\in \tilde{V}_\mu$ and $Y\perp Y_1$}.
\end{equation}
Then, for any $Y\in \tilde{V}_\lambda$ and $Y\perp X_1$, by $Y,PY\in V^\perp$,  we have 
\begin{equation}\label{eqn:5.1a5g1}
\begin{aligned}
\langle TX_1,TY \rangle=\langle TX_1,PY \rangle=\langle PX_1,PY \rangle=0.
\end{aligned}
\end{equation}
By \eqref{eqn:5.1f2}, we have 
$$
\langle TX_1,TY \rangle=\langle TX_1,Y_1 \rangle\langle Y_1,TY \rangle=p_2\langle TY_1,Y \rangle, 
$$
it follows from \eqref{eqn:5.1a5g1} and $p_2\neq0$ that $\langle TY_1,Y \rangle=0$ 
for any $Y\in \tilde{V}_\lambda$ and $Y\perp X_1$.
So, we still have $T({\rm Span}\{X_1,Y_1\})\subset {\rm Span}\{X_1,Y_1\}$. Thus, on $\Omega$, it always holds that $T({\rm Span}\{X_1,Y_1\})\subset {\rm Span}\{X_1,Y_1\}$. 
	
Since ${\rm Span}\{\bar{V},\bar{X}\}={\rm Span}\{X_1,Y_1\}$ and 
$T({\rm Span}\{X_1,Y_1\})\subset {\rm Span}\{X_1,Y_1\}$, we can conclude that $T({\rm Span}\{\bar{V},\bar{X}\})\subset {\rm Span}\{\bar{V},\bar{X}\}$. 
Then by $\bar{X},T\bar{X}\in V^\perp$, $T\bar{V}=-C\bar{V}$ and  $\|T\bar{X}\|=1$ we have $T\bar{X}=\pm \bar{X}$. Now, by \eqref{eqn:5.1a1d1}, \eqref{eqn:5.1a1d2}, $T\bar{V}=-C\bar{V}$ 
and $T\bar{X}=\pm \bar{X}$, we have ${\rm Tr}P=\pm 1$. By ${\rm Tr}P=m-n$, we get $m-n=\pm 1$. Interchanging the two factors if necessary, we may assume $m\geq n$, then $T\bar{X}=\bar{X}$. 
Since $T({\rm Span}\{X_1,Y_1\})\subset {\rm Span}\{X_1,Y_1\}$, then by \eqref{eqn:5.1a1d1} and \eqref{eqn:5.1a1d2}, we have $P(\tilde{V}_\lambda\ominus\mathbb{R}X_1)=\tilde{V}_\mu\ominus\mathbb{R}Y_1$. 
	
In the following, we can take a local orthonormal frame field $\{E_i\}_{i=1}^{m+n-1}$, such that 
\begin{equation}\label{eqn:5.1a5g1g1s1}
\begin{aligned}
&E_1=\bar{V},\ E_2=\sin(\theta)X_1-\cos(\theta)Y_1,\ E_{2k+1}\in \tilde{V}_\lambda,\ E_{2k+2}=PE_{2k+1}\in \tilde{V}_\mu,\ \ 
1\leq k\leq n-1. \\ 
&AE_1=(\cos^2(\theta)\lambda+\sin(\theta)^2\mu)E_1+\sin(\theta)\cos(\theta)(\lambda-\mu)E_2,\\ 
&AE_2=\sin(\theta)\cos(\theta)(\lambda-\mu)E_1+(\sin^2(\theta)\lambda+\cos^2(\theta)\mu)E_2,\\ 
&PE_1=-CE_1+\sqrt{1-C^2}N, \ PE_2=E_2. 
\end{aligned}
\end{equation}
Assume that $\nabla_{E_i}{E_j}=\Gamma_{i,j}^kE_k$, $1\leq i,j,k\leq m+n-1$. 
	
Now, by applying the parallelism of the product structure $P$, 
i.e. $\bar{\nabla}P=0$, we can obtain some information about the connection coefficients with respect to $\{E_i\}_{i=1}^{m+n-1}$. Specifically, we calculate that 
\begin{equation}\label{eqn:5.1a6}
\begin{aligned}
\langle (\bar{\nabla}_{E_{2k+1}}P)E_1,E_{2k+1}\rangle
&=\langle \bar{\nabla}_{E_{2k+1}}(-CE_1+\sqrt{1-C^2}N), E_{2k+1}\rangle-\langle \bar{\nabla}_{E_{2k+1}} E_1,P E_{2k+1}\rangle\\
&=-C\Gamma_{2k+1,1}^{2k+1}-\sqrt{1-C^2}\lambda-\Gamma_{2k+1,1}^{2k+2}, 
\end{aligned}
\end{equation}
and 
\begin{equation}\label{eqn:5.1a7}
\begin{aligned}
\langle (\bar{\nabla}_{E_{2k+1}}P)E_1,E_{2k+2}\rangle
&=\langle \bar{\nabla}_{E_{2k+1}}(-CE_1+\sqrt{1-C^2}N), E_{2k+2}\rangle-\langle \bar{\nabla}_{E_{2k+1}} E_1,P E_{2k+2}\rangle\\
&=-C\Gamma_{2k+1,1}^{2k+2}-\Gamma_{2k+1,1}^{2k+1}. 
\end{aligned}
\end{equation}
It follows from \eqref{eqn:5.1a6}, \eqref{eqn:5.1a7} and $\bar{\nabla}P=0$ that 
\begin{equation}\label{eqn:5.1a8}
\Gamma_{2k+1,1}^{2k+1}=\frac{C\lambda}{\sqrt{1-C^2}},\ 
\Gamma_{2k+1,1}^{2k+2}=\frac{-\lambda}{\sqrt{1-C^2}}. 
\end{equation}
We calculate that 
\begin{equation}\label{eqn:5.1a9}
\begin{aligned}
\langle (\bar{\nabla}_{E_{2k+2}}P)E_1,E_{2k+1}\rangle
&=\langle \bar{\nabla}_{E_{2k+2}}(-CE_1+\sqrt{1-C^2}N), E_{2k+1}\rangle-\langle \bar{\nabla}_{E_{2k+2}} E_1,P E_{2k+1}\rangle\\
&=-C\Gamma_{2k+2,1}^{2k+1}-\Gamma_{2k+2,1}^{2k+2}, 
\end{aligned}
\end{equation}
and 
\begin{equation}\label{eqn:5.1a10x1}
\begin{aligned}
\langle (\bar{\nabla}_{E_{2k+2}}P)E_1,E_{2k+2}\rangle
&=\langle \bar{\nabla}_{E_{2k+2}}(-CE_1+\sqrt{1-C^2}N), E_{2k+2}\rangle-\langle \bar{\nabla}_{E_{2k+2}} E_1,P E_{2k+2}\rangle\\
&=-C\Gamma_{2k+2,1}^{2k+2}-\sqrt{1-C^2}\mu-\Gamma_{2k+2,1}^{2k+1}. 
\end{aligned}
\end{equation}
It follows from \eqref{eqn:5.1a9}, \eqref{eqn:5.1a10x1} and $\bar{\nabla}P=0$ that 
\begin{equation}\label{eqn:5.1a11}
\Gamma_{2k+2,1}^{2k+1}=\frac{-\mu}{\sqrt{1-C^2}},\ 
\Gamma_{2k+2,1}^{2k+2}=\frac{C\mu}{\sqrt{1-C^2}}. 
\end{equation}
Similarly, by calculating 
$$
\langle (\bar{\nabla}_{E_1}P)E_{2k+1},E_{2k+1}\rangle=0,\ 
\langle (\bar{\nabla}_{E_{2k+1}}P)E_2,E_{2k+1}\rangle=0,\   
\langle (\bar{\nabla}_{E_{2k+2}}P)E_2,E_{2k+1}\rangle=0,
$$
we can obtain that 
\begin{equation}\label{eqn:5.1a12}
\Gamma_{1,2k+1}^{2k+2}=0,\ \ \Gamma_{2k+1,2}^{2k+2}=\Gamma_{2k+1,2}^{2k+1},\ \ 
\Gamma_{2k+2,2}^{2k+2}=\Gamma_{2k+2,2}^{2k+1}. 
\end{equation}
	
In the following, we take $(X,Y,Z)=(E_1,E_{2k+1},E_{2k+1})$ into Codazzi equation \eqref{eqn:2.3}, then we get 
\begin{equation}\label{eqn:5.1a13}
\begin{aligned}
\langle& (\nabla_{E_{1}}A)E_{2k+1},E_{2k+1}\rangle
-\langle (\nabla_{E_{2k+1}}A)E_{1},E_{2k+1}\rangle\\
&=-\frac{1}{2}(\langle E_1,V \rangle\langle TE_{2k+1},E_{2k+1} \rangle-\langle E_{2k+1},V \rangle\langle TE_{1},E_{2k+1} \rangle)=0. 
\end{aligned}
\end{equation}
On the other hand, by definition of $\nabla A$, we have 
\begin{equation}\label{eqn:5.1a14}
\begin{aligned}
\langle& (\nabla_{E_{1}}A)E_{2k+1},E_{2k+1}\rangle
-\langle (\nabla_{E_{2k+1}}A)E_{1},E_{2k+1}\rangle\\
&=-\langle {\nabla}_{E_{2k+1}}\Big((\cos^2(\theta)\lambda+\sin(\theta)^2\mu)E_1+\sin(\theta)\cos(\theta)(\lambda-\mu)E_2\Big),E_{2k+1}\rangle+\lambda\langle \nabla_{E_{2k+1}} E_1, E_{2k+1}\rangle\\
&=(\lambda-\mu)\sin(\theta)(\sin(\theta)\Gamma_{2k+1,1}^{2k+1}-\cos(\theta)\Gamma_{2k+1,2}^{2k+1}). 
\end{aligned}
\end{equation}
It follows from \eqref{eqn:5.1a13} and \eqref{eqn:5.1a14} that  
\begin{equation}\label{eqn:5.1a15}
\sin(\theta)\Gamma_{2k+1,1}^{2k+1}-\cos(\theta)\Gamma_{2k+1,2}^{2k+1}=0.
\end{equation}
	
Now, by taking $(X,Y,Z)=(E_1,E_{2k+1},E_{2k+2})$ in Codazzi equation \eqref{eqn:2.3}, we have 
\begin{equation}\label{eqn:5.1a16}
\begin{aligned}
\langle& (\nabla_{E_{1}}A)E_{2k+1},E_{2k+2}\rangle
-\langle (\nabla_{E_{2k+1}}A)E_{1},E_{2k+2}\rangle\\
&=-\frac{1}{2}(\langle E_1,V \rangle\langle TE_{2k+1},E_{2k+2} \rangle-\langle E_{2k+1},V \rangle\langle TE_{1},E_{2k+2} \rangle)=-\frac{1}{2}\sqrt{1-C^2}. 
\end{aligned}
\end{equation}
On the other hand, by definition of $\nabla A$, we have 
\begin{equation}\label{eqn:5.1a17}
\begin{aligned}
\langle &(\nabla_{E_{1}}A)E_{2k+1},E_{2k+2}\rangle
-\langle (\nabla_{E_{2k+1}}A)E_{1},E_{2k+2}\rangle\\
&=(\lambda-\mu)\Gamma_{1,2k+1}^{2k+2}-\langle {\nabla}_{E_{2k+1}}((\cos^2(\theta)\lambda+\sin(\theta)^2\mu)E_1+\sin(\theta)\cos(\theta)(\lambda-\mu)E_2),E_{2k+2}\rangle\\
&\quad\quad  +\mu\langle \nabla_{E_{2k+1}} E_1, E_{2k+2}\rangle\\
&=(\lambda-\mu)\Gamma_{1,2k+1}^{2k+2}-(\cos^2(\theta)\lambda+\sin(\theta)^2\mu)\Gamma_{2k+1,1}^{2k+2}-
\sin(\theta)\cos(\theta)(\lambda-\mu)\Gamma_{2k+1,2}^{2k+2}+\mu 
\Gamma_{2k+1,1}^{2k+2}
\end{aligned}
\end{equation}
It follows from \eqref{eqn:5.1a16} and \eqref{eqn:5.1a17} that
\begin{equation}\label{eqn:5.1a18}
(\lambda-\mu)\Gamma_{1,2k+1}^{2k+2}-(\cos^2(\theta)\lambda+\sin(\theta)^2\mu)
\Gamma_{2k+1,1}^{2k+2}-\sin(\theta)\cos(\theta)(\lambda-\mu)\Gamma_{2k+1,2}^{2k+2}+\mu 
\Gamma_{2k+1,1}^{2k+2}=-\frac{1}{2}\sqrt{1-C^2}.
\end{equation}
	
Similarly, by taking $(X,Y,Z)=(E_1,E_{2k+2},E_{2k+1})$ and $(X,Y,Z)=(E_1,E_{2k+2},E_{2k+2})$ into Codazzi equation \eqref{eqn:2.3}, we have 
\begin{equation}\label{eqn:5.1a19}
(\lambda-\mu)\Gamma_{1,2k+1}^{2k+2}-(\cos^2(\theta)\lambda+\sin(\theta)^2\mu)\Gamma_{2k+2,1}^{2k+1}-\sin(\theta)\cos(\theta)(\lambda-\mu)\Gamma_{2k+2,2}^{2k+1}+\lambda 
\Gamma_{2k+2,1}^{2k+1}=-\frac{1}{2}\sqrt{1-C^2}.
\end{equation}
and 
\begin{equation}\label{eqn:5.1a20}
\cos(\theta)\Gamma_{2k+2,1}^{2k+2}+\sin(\theta)\Gamma_{2k+2,2}^{2k+2}=0.
\end{equation}
	
By using \eqref{eqn:5.1a8}, \eqref{eqn:5.1a11}, \eqref{eqn:5.1a15} and \eqref{eqn:5.1a20}, we have $\Gamma_{2k+1,2}^{2k+1}=\frac{C\lambda\tan(\theta)}{\sqrt{1-C^2}}$ and $\Gamma_{2k+2,2}^{2k+1}=\frac{-C\mu\cot(\theta)}{\sqrt{1-C^2}}$. Then, by 
\eqref{eqn:5.1a8}, \eqref{eqn:5.1a11} and \eqref{eqn:5.1a12}, 
we calculate \eqref{eqn:5.1a18}$\times\mu$--\eqref{eqn:5.1a19}$\times\lambda$, and obtain $(C-1)(\lambda-\mu)(1+C-2\lambda\mu)=0$. It follows from $C\neq\pm1$ and 
$\lambda\neq\mu$ that $C=-1+2\lambda\mu$. It implies that $C$ is constant on $M_{22}$ and $AV=0$, which contradicts the definition of $M_{22}$. Thus $M_{22}$ is empty, and  $M_2=M_{21}$. 
	
In the following, we assume $AV=\lambda V$ on $M_2$. If ${\rm dim}\tilde{V}_\lambda=1$, then it holds $V^\perp=\tilde{V}_\mu$, which implies that $M$ satisfies $AT=TA$. For any $X,Y\in \tilde{V}_\mu$, by Codazzi equation \eqref{eqn:2.3}, we have 
\begin{equation}\label{eqn:5.1azz1}
\begin{aligned}
\langle (\nabla_{X}A)V,Y\rangle-\langle (\nabla_{V}A)X,Y\rangle
=(\lambda-\mu)\langle \nabla_{X} V,Y\rangle
=\frac{1}{2}(1-C^2)\langle TX,Y\rangle.
\end{aligned}
\end{equation}
By \eqref{eqn:2.4}, we have $\nabla_{X} V=CAX-TAX=\mu(CX-TX)$. It follows from \eqref{eqn:5.1azz1} that 
\begin{equation}\label{eqn:5.1azz2}
\begin{aligned}
(\lambda-\mu)\mu[C\langle X,Y\rangle-\langle TX,Y\rangle]=\frac{1}{2}(1-C^2)\langle TX,Y\rangle.
\end{aligned}
\end{equation}
	
Let $\mathcal D_+=\{X\in V^\perp|\ TX=X\}$ and $\mathcal D_{-}=\{X\in V^\perp|\ TX=-X\}$. 
Then we have ${\rm dim}\mathcal D_+=m-1\geq2$ and ${\rm dim}\mathcal D_{-}=n-1\geq1$. 
Now, for any $\tilde{X}\in V^\perp$ such that $P\tilde{X}=\tilde{X}$, we take 
$Y=X=\tilde{X}$ into \eqref{eqn:5.1azz2}, then 
we have $(\lambda-\mu)\mu=-\frac{1}{2}(1+C)$. 
On the other hand, for any $\tilde{X}\in V^\perp$ such that $P\tilde{X}=-\tilde{X}$, we take 
$Y=X=\tilde{X}$ into \eqref{eqn:5.1azz2}, then 
we have $(\lambda-\mu)\mu=-\frac{1}{2}(1-C)$. Thus, it holds $C=0$ on $M_2$. 
	
Now, we assume that ${\rm dim}\tilde{V}_\lambda\geq2$. For any unit vector fields $\tilde{X},\tilde{Y}\in (\tilde{V}_\lambda\ominus V)$, by Codazzi equation \eqref{eqn:2.3}, we have 
\begin{equation}\label{eqn:5.1a21}
\begin{aligned}
\langle& (\nabla_{\bar{V}}A)\tilde{X},\tilde{Y}\rangle
-\langle (\nabla_{\tilde{X}}A)\bar{V},\tilde{Y}\rangle\\
&=-\frac{1}{2}(\langle \bar{V},V \rangle\langle T\tilde{X},\tilde{Y} \rangle-\langle \tilde{X},V \rangle\langle T\bar{V},\tilde{Y} \rangle)=-\frac{1}{2}\sqrt{1-C^2}\langle T\tilde{X},\tilde{Y} \rangle. 
\end{aligned}
\end{equation}
On the other hand, by \eqref{eqn:5.1a1g1}, 
we have $\langle (\nabla_{\bar{V}}A)\tilde{X},\tilde{Y}\rangle
-\langle (\nabla_{\tilde{X}}A)\bar{V},\tilde{Y}\rangle=0$. 
It follows that $\langle T\tilde{X},\tilde{Y}\rangle=0$ for 
any $\tilde{X},\tilde{Y}\in (\tilde V_\lambda\ominus V)$. Thus, $T\tilde{X}\in \tilde{V}_\mu$ for any $\tilde{X}\in (\tilde{V}_\lambda\ominus V)$.  
	
Now, for any unit vector field $\tilde{X}\in (\tilde V_\lambda\ominus V)$, by Codazzi equation \eqref{eqn:2.3} and the definition of $\nabla A$, we have 
\begin{equation}\label{eqn:5.1a22}
\begin{aligned}
\langle (\nabla_{T\tilde{X}}A)\bar{V},\tilde{X}\rangle
-\langle (\nabla_{\bar{V}}A)T\tilde{X},\tilde{X}\rangle
=-(\mu-\lambda)\langle \nabla_{\bar{V}} T\tilde{X},\tilde{X}\rangle
=\frac{1}{2}\sqrt{1-C^2}. 
\end{aligned}
\end{equation} 
On the other hand, we calculate that 
\begin{equation}\label{eqn:5.1a23}
\begin{aligned}
0&=\langle (\bar{\nabla}_{\bar{V}}P)\tilde{X},\tilde{X}\rangle
=\langle \bar{\nabla}_{\bar{V}}T\tilde{X},\tilde{X}\rangle-\langle \bar{\nabla}_{\bar{V}}\tilde{X},T\tilde{X}\rangle
=2\langle \nabla_{\bar{V}}T\tilde{X},\tilde{X}\rangle, 
\end{aligned}
\end{equation}
where we have used that $\tilde{X}\in \tilde{V}_\lambda$ and $T\tilde{X}\in \tilde{V}_\mu$. 
Combining \eqref{eqn:5.1a22} and \eqref{eqn:5.1a23}, we have 
$\frac{1}{2}\sqrt{1-C^2}=0$, which contradicts the assumption of $M_2$. Thus, it always holds 
that $AT=TA$ and $C=0$ on $M_2$. 
	
According to the continuity of function $C$ and the connectivity of hypersurface $M$, it holds that either $C=0$ on $M$ or $C=\pm1$ on $M$. When $C=\pm1$, by Lemmas \ref{lemma:3.5}--\ref{lemma:3.2aa},  $M$ is either an open part of $\Sigma\times \mathbb{H}^n$, where $\Sigma$ is a totally umbilical hypersurface of $\mathbb{H}^m$ (see Example \ref{exam:1}) or an open part of $\mathbb{H}^m\times \tilde{\Sigma}$, where $\tilde{\Sigma}$ is a totally umbilical hypersurface of $\mathbb{H}^n$ (see Example \ref{exam:2}). When $C=0$, then it holds $AT=TA$ on $M$. According to Proposition \ref{prop:ATTA} and the fact that $M$ has two distinct constant principal curvatures, we know that $M$ is an open part of 
$M_{1,-1}^0$.  	
\end{proof}
	
\begin{remark}\label{rem:1.1}
In Theorem \ref{thm:1.2}, the assumption of constant principal curvatures cannot be omitted. In fact, we can construct some hypersurfaces with two nonconstant principal curvatures (see Example \ref{ex:two-pc-horospheres}).
\end{remark}

\section{Hypersurfaces of $\mathbb{H}^m\times \mathbb{H}^n$ with constant principal curvatures and constant product angle function}\label{sect:6} 
	
In this section, we classify the hypersurfaces in $\mathbb{H}^m\times\mathbb{H}^n$ ($m\geq3, n\geq2$) with constant principal curvatures and constant product angle function $C$. First, we consider the case where $C$ is not equal to $0$ or $\pm1$. 
	
\begin{proposition}\label{prop:Cneq0} 
Let $M$ be a connected oriented hypersurface of $\mathbb{H}^m\times \mathbb{H}^n$ 
($m\geq3, n\geq2$) with constant principal curvatures and constant product angle function $C\notin\{0,\pm1\}$. Then $g=3$, and up to isometries of $\mathbb{H}^m\times \mathbb{H}^n$, one of the following two cases occurs: 
\begin{itemize}
\item[(1)] $M$ is an open part of $M_{1,-1}^{C}$ (see Example \ref{exam:3}); or
\item[(2)] $M$ is an open part of $M_{1,1}^{C}$ (see Example \ref{exam:4}). 
\end{itemize}
\end{proposition}
	
\begin{proof}
When $C\neq 0$, by Remark \ref{rem:denominater}, it holds $\lambda^2-\frac{1}{2}(1-C^2)\neq 0$ for any principal curvature $\lambda\in \sigma(A|_{V^\perp})$. Then   Lemma \ref{lemma:4.1} says
\begin{equation}\label{eqn:6.1}
\langle TX,Y\rangle=\frac{C\lambda^2}{\lambda^2-\frac{1}{2}(1-C^2)}\langle X,Y  \rangle, \ \ \forall \ X,Y\in V_\lambda.   
\end{equation}
		
If $\lambda \in \left\{ \pm \sqrt{\frac{1+ C}{2}}, \pm \sqrt{\frac{1 - C}{2}}\right\}$, then $\frac{C\lambda^2}{\lambda^2-\frac{1}{2}(1-C^2)} = \pm 1$. It follows from \eqref{eqn:6.1} that $TV_\lambda = V_\lambda$. 
		
Set $E := \sigma(A|_{V^\perp}) \backslash \left\{0, \pm \sqrt{\frac{1+ C}{2}}, \pm \sqrt{\frac{1 - C}{2}}\right\}$. 
If $E$ is not empty, we can take a special $\tilde{\lambda}\in E$ such that there does not exist any $\lambda\in E$ lying between $\tilde{\lambda}$ and $\frac{1 - C^2}{2\tilde{\lambda}}$. Without loss of generality, we further suppose that $\tilde{\lambda} > 0$. 
Substituting this $\tilde{\lambda}$ into Lemma \ref{lemma:car-seperate} gives
\begin{equation}\label{eqn:6.1f1}
\sum_{i=1, \mu_i\neq\tilde\lambda}^{m+n-2}\frac{\tilde{\lambda}\mu_i-\frac{1}{2}(1-C^2)}{\tilde{\lambda}-\mu_i}\langle TX,e_i\rangle^2=0.
\end{equation}
If $\mu_i \in \left\{\pm \sqrt{\frac{1+ C}{2}}, \pm \sqrt{\frac{1 - C}{2}}\right\}$, then by $TV_{\mu_i} = V_{\mu_i}$, we have $\langle TX,e_i\rangle= 0$ for any $X\in V_{\tilde{\lambda}}$ and $e_i\in V_{\mu_i}$. So \eqref{eqn:6.1f1} reduces to 
\begin{equation}\label{eqn:6.1f2}
\sum_{i=1, \mu_i\neq\tilde{\lambda},\mu_i\in E\cup\{0\}}^{m+n-2}\frac{\tilde{\lambda}\mu_i-\frac{1}{2}(1-C^2)}{\tilde{\lambda}-\mu_i}\langle TX,e_i\rangle^2=0.
\end{equation}
		
For any $\mu_i\in E\cup\{0\}$, $\mu_i\neq \tilde{\lambda}$, the coefficient
$$
\frac{\tilde{\lambda}\mu_i-\frac{1}{2}(1-C^2)}{\tilde{\lambda}-\mu_i} = \tilde{\lambda}\cdot \frac{\mu_i-\frac{1-C^2}{2\tilde{\lambda}}}{\tilde{\lambda}-\mu_i}
$$ 
is negative unless $\mu_i = \frac{1 - C^2}{2\tilde{\lambda}}$. Here $\frac{1 - C^2}{2\tilde{\lambda}} \neq \tilde{\lambda}$ since $\tilde{\lambda}^2-\frac{1}{2}(1-C^2)\neq 0$. By $\tilde{\lambda}\in E$, then for any $X\in V_{\tilde{\lambda}}$, $TX$ has nontrivial projection on $(\oplus_{\mu_i\in E\cup\{0\}}V_{\mu_i})\ominus V_{\tilde{\lambda}}$. If $\frac{1-C^2}{2\tilde{\lambda}} \notin E$, then the left hand side of \eqref{eqn:6.1f2} is negative, which is a contradiction. Thus, $\frac{1-C^2}{2\tilde{\lambda}} \in E$. Again by \eqref{eqn:6.1f2}, we must have $TV_{\tilde{\lambda}} \subset (V_{\tilde{\lambda}}\oplus V_{\frac{1-C^2}{2\tilde{\lambda}}})$. Let $\tilde{\eta}=\frac{1-C^2}{2\tilde{\lambda}}$. 
		
On the other hand, by substituting $\tilde{\eta}$ into Lemma \ref{lemma:car-seperate}, we can also obtain $TV_{\tilde{\eta}}\subset V_{\tilde\lambda}\oplus V_{\tilde{\eta}}$. It follows that $T\br{V_{\tilde{\lambda}}\oplus V_{\tilde{\eta}}} = V_{\tilde{\lambda}}\oplus V_{\tilde{\eta}}$. 
		
For any unit $\tilde X\in V_{\tilde{\lambda}}$, by \eqref{eqn:6.1} and $T\br{V_{\tilde{\lambda}}\oplus V_{\tilde{\eta}}} = V_{\tilde{\lambda}}\oplus V_{\tilde{\eta}}$, we can assume 
\begin{equation}\label{eqn:6.1f3}
T\tilde X = a \tilde X + b \tilde Y,\quad a = \frac{C\tilde{\lambda}^2}{\tilde{\lambda}^2-\frac{1}{2}(1-C^2)},\ a^2 + b^2 = 1, \ b\neq 0,
\end{equation}
where $\tilde Y\in V_{\tilde{\eta}}$ is a unit vector field. Then by \eqref{eqn:6.1f3}, we get 
$$
bT\tilde Y = \tilde X - aT\tilde X = (1-a^2)\tilde X - ab\tilde Y = b^2 \tilde X - ab \tilde Y,
$$
i.e.,
$$
T\tilde Y = -a \tilde Y + b \tilde X.
$$
For unit vector field $\tilde Y\in V_{\tilde{\eta}}$, by \eqref{eqn:6.1}, we have $-a=\lang{T\tilde Y, \tilde Y} = \frac{C\tilde{\eta}^2}{\tilde{\eta}^2-\frac{1}{2}(1-C^2)}$. It follows 
\begin{equation}\label{eqn:6.1f4}
\frac{C\tilde{\lambda}^2}{\tilde{\lambda}^2-\frac{1}{2}(1-C^2)} + \frac{C\tilde{\eta}^2}{\tilde{\eta}^2-\frac{1}{2}(1-C^2)} = 0.
\end{equation}
Substituting $2\tilde{\lambda}\tilde{\eta}= 1-C^2$ into \eqref{eqn:6.1f4}, we obtain
$$
0 = \frac{C\tilde{\lambda}^2}{\tilde{\lambda}^2-\tilde{\lambda}\tilde{\eta}} + \frac{C\tilde{\eta}^2}{\tilde{\eta}^2-\tilde{\lambda}\tilde{\eta}} =C,
$$
which contradicts $C\neq0$. Thus $E$ is empty, and $\sigma(A|_{V^\perp})\subset \left\{0, \pm \sqrt{\frac{1+ C}{2}}, \pm \sqrt{\frac{1 - C}{2}}\right\}$. In the following, we divide the discussion into three cases:
		
{\bf Case-1}: $\sigma(A|_{V^\perp}) = \{0\}$. 
		
{\bf Case-2}: $0\in \sigma(A|_{V^\perp})$ and $\sigma(A|_{V^\perp})\backslash\{0\}$ is not empty.  
		
{\bf Case-3}: $\sigma(A|_{V^\perp})\subset \left\{\pm \sqrt{\frac{1+ C}{2}}, \pm \sqrt{\frac{1 - C}{2}}\right\}$. 
		
In {\bf Case-1}, $M$ is a totally geodesic hypersurface. According to the classification result in \cite{Toj13}, $M$ has constant product angle function $C\equiv \pm 1$, which contradicts the assumption. 
		
		
In {\bf Case-2}, by Corollary \ref{coro:4.2},  it holds $TV_0 \perp V_0$, so $TV_0\subset (V^{\perp}\ominus V_0)$. On the other hand, for any $\lambda\in \left\{\pm \sqrt{\frac{1+ C}{2}}, \pm \sqrt{\frac{1 - C}{2}}\right\}$, we have $TV_\lambda=V_\lambda$, hence $T(V^{\perp}\ominus V_0)=(V^{\perp}\ominus V_0)$. We get a contradiction. 
		
		
In {\bf Case-3}, for any $\lambda \in \sigma(A|_{V^\perp})$, we have $TV_\lambda = V_\lambda$. It means that $AT = TA$ holds on $M$. Then by Proposition \ref{prop:ATTA}, 
$M$ is either an open part of $M_{1,-1}^{C}$ for some $C\in (-1,0)\cup(0,1)$, or 
an open part of $M_{1,1}^{C}$ for some $C\in (-1,0)\cup(0,1)$. 
\end{proof}
	
Next, we consider the case when $C = 0$.
	
\begin{proposition}
Let $M$ be a connected oriented hypersurface of $\mathbb{H}^m\times \mathbb{H}^n$ ($m\geq3, n\geq2$) with constant principal curvatures and product angle function $C\equiv0$. Then $g\leq 3$.
\end{proposition}
	
\begin{proof}
Let $\tilde{g}$ denote the number of elements in $\sigma(A|_{V^\perp})$. In the following, we divide the discussion into three cases, according to the intersection of $\sigma(A|_{V^\perp})$ with $\{\pm1/\sqrt2\}$: 
		
{\bf Case-i}: $\pm \frac{1}{\sqrt{2}}\notin \sigma(A|_{V^\perp})$; 
		
{\bf Case-ii}: one of $\{\trh, -\trh\}$ belongs to $\sigma(A|_{V^\perp})$;   
		
{\bf Case-iii}: $\{\trh, -\trh\}\subset\sigma(A|_{V^\perp})$. 
		
In {\bf Case-i}, for any $\lambda\in\sigma(A|_{V^\perp})$, by Corollary \ref{coro:4.2} (2), we have $TV_\lambda \perp V_\lambda$. When $\tilde{g}=1$, then $g\leq2$. 
When $\tilde{g}\geq 2$, we can take a specific $\tilde{\lambda}\in \sigma(A|_{V^\perp})$, $\tilde{\lambda} \neq 0$, such that there is no other element $\mu_i\in \sigma(A|_{V^\perp})$ lying between $\tilde{\lambda}$ and $\frac{1}{2\tilde{\lambda}}$. Up to a sign of unit vector field of $N$, we can assume $\tilde{\lambda}>0$. 
For such $\tilde{\lambda}$, by \eqref{eqn:car-0-seperate-1}, we get 
\begin{equation}\label{eqn:6.as1}
0 = \sum_{i=1,\mu_i\neq \tilde{\lambda}}^{m+n-2}\frac{\tilde{\lambda}\mu_i - \frac{1}{2}}{\tilde{\lambda} - \mu_i} =\tilde{\lambda} \sum_{i=1,\mu_i\neq \tilde{\lambda}}^{m+n-2}\frac{\mu_i - \frac{1}{2\tilde{\lambda}}}{\tilde{\lambda}- \mu_i}.
\end{equation}
Since for any $\mu_i\in \sigma(A|_{V^\perp})$, $\mu_i\neq\tilde{\lambda}$, it holds $\frac{\mu_i - \frac{1}{2\tilde{\lambda}}}{\tilde{\lambda}- \mu_i}\leq0$. Then by  \eqref{eqn:6.as1}, it follows that  $\sigma(A|_{V^\perp}) = \{\lambda, \frac{1}{2\lambda}\}$, which means $\tilde{g}=2$ and $g\leq3$. 
		
		
In {\bf Case-ii}, up to a sign of unit vector field $N$, we can suppose $\trh \in \sigma(A|_{V^\perp})$ and $-\trh \notin \sigma(A|_{V^\perp})$. 
If ${\rm dim}\br{V_{\trh}} \geq m+n-3$, it follows from ${\rm dim}\br{V^\perp}=m+n-2$ that 
$\tilde{g}\leq 2$, then $g\leq3$. 
So we only need to consider the situation that ${\rm dim}\br{V_{\trh}} \leq m+n-4$. 
		
By Corollary \ref{coro:4.2} (2), we have $TV_\lambda \perp V_\lambda$ for any  $\lambda\neq \trh$ and $\lambda\in \sigma(A|_{V^\perp})$. For any unit vector field $X\in V_\trh$, by \eqref{eqn:car-0-general},  we have 
\begin{align*}
0 = & \sum_{i=1, \mu_i\neq\frac{1}{\sqrt{2}}}^{m+n-2}\frac{\frac{1}{\sqrt{2}}\mu_i-\frac{1}{2}}{\frac{1}{\sqrt{2}}-\mu_i} \big(1+\langle TX,X\rangle\langle Te_i,e_i\rangle-2\langle TX,e_i\rangle^2\big)
= \sum_{i=1, \mu_i\neq\frac{1}{\sqrt{2}}}^{m+n-2}-\trh \big(1 - 2\langle TX,e_i\rangle^2\big)\\
= & -\trh \br{m+n-2- {\rm dim}\br{V_\trh} -2\sum_{i=1, \mu_i\neq\frac{1}{\sqrt{2}}}^{m+n-2} \langle TX,e_i\rangle^2}.
\end{align*}
Therefore, we have
$$
m+n-2- {\rm dim}\br{V_\trh} = 2\sum_{i=1, \mu_i\neq\frac{1}{\sqrt{2}}}^{m+n-2}\langle TX,e_i\rangle^2\leq 2|TX|^2 = 2.
$$
By the assumption that $ {\rm dim}\br{V_\trh} \leq m+n-4$, we have ${\rm dim}\br{V_\trh} =m+n-4$, and
\begin{equation}\label{eqn:6.2}
\sum_{i=1, \mu_i\neq\frac{1}{\sqrt{2}}}^{m+n-2}\langle TX,e_i\rangle^2 = |TX|^2 = 1, \
\text{for any unit $X\in V_\trh$}. 
\end{equation} 
Hence $TV_{\frac{1}{\sqrt{2}}}\subset  (V^\perp\ominus V_{\frac{1}{\sqrt{2}}})$. By ${\rm dim}\br{V_\trh} =m+n-4$ and ${\rm dim}(V^\perp\ominus V_{\frac{1}{\sqrt{2}}})=2$, we have ${\rm dim}\br{V_\trh}\leq2$,  $m+n\leq6$ and $\tilde{g}\leq3$. 
		
Suppose, to the contrary, that $\tilde{g}= 3$.  Then $\sigma(A|_{V^\perp}) = \{\lambda, \mu, \trh\}$, where $\lambda^2\neq \frac{1}{2}$,  $\mu^2\neq \frac{1}{2}$ and $\lambda\neq\mu$. 
It also holds that ${\rm dim}\br{V_{\lambda}} ={\rm dim}\br{V_{\mu}} =1$. By Corollary \ref{coro:4.2} (2), we have $TV_\lambda\perp V_\lambda$ and $TV_\mu\perp V_\mu$. 
		
If ${\rm dim}\br{V_\trh} = 2$, by $TV_{\frac{1}{\sqrt{2}}}\perp V_{\frac{1}{\sqrt{2}}}$, $TV_\lambda\perp V_\lambda$ and $TV_\mu\perp V_\mu$,  
we can choose an adapted principal orthonormal basis of $V^\perp$ so  that 
$$
Ae_1 = \lambda e_1,\ Ae_2 = \mu e_2,\ Ae_3 = \trh e_3, Ae_4 = \trh e_4,
$$
and 
\begin{equation}\label{eqn:6.as2}
T\left(
\begin{array}{c}
e_1 \\
e_2 \\
e_3 \\
e_4 \\
\end{array}
\right)= 
\left(
\begin{array}{cccc}
0 & p_{12} & p_{13} & 0\\
p_{12} & 0 & p_{23} & p_{24}\\
p_{13} & p_{23} & 0 & 0\\
0 & p_{24} & 0 & 0
\end{array}
\right)
\left(
\begin{array}{c}
e_1 \\
e_2 \\
e_3 \\
e_4 \\
\end{array}
\right).
\end{equation}
Because $V^\perp$ is $P$-invariant, $T=P$ on $V^\perp$, so we have 
\begin{equation}\label{eqn:6.as3}
p_{24}^2 = 1,\ \  p_{12} = p_{23} = 0,\ \   p_{13}^2 = 1.
\end{equation}
Without loss of generality, suppose $\lambda\neq 0$. For this $\lambda\neq 0$ and $\lambda\neq \pm\trh$, by using \eqref{eqn:6.as2} and \eqref{eqn:6.as3}, it follows from \eqref{eqn:car-0-seperate-2} that 
$$
0 =\frac{\lambda\mu - \frac{1}{2}}{\lambda - \mu}\langle Te_1,e_2\rangle^2+\frac{\lambda\frac{1}{\sqrt{2}}- \frac{1}{2}}{\lambda -\frac{1}{\sqrt{2}}}(\langle Te_1,e_3\rangle^2+\langle Te_1,e_4\rangle^2)=  \frac{\lambda\trh - \frac{1}{2}}{\lambda - \trh} = \trh,
$$
which is a contradiction. 
		
If ${\rm dim}\br{V_\trh} = 1$, we have $m+n=5$ and ${\rm Tr}P=m-n\neq0$. On the other hand, by $PN=CN+V$ and $PV=-CV+(1-C^2)N$, $TV_\trh \perp V_\trh$, $TV_\lambda \perp V_\lambda$ and $TV_\mu \perp V_\mu$, we obtain ${\rm Tr}P=0$, which is a contradiction. 
		
In conclusion, when $\trh \in \sigma(A|_{V^\perp})$ and $-\trh \notin \sigma(A|_{V^\perp})$, we have $\widetilde g\le2$, and hence $g\le3$.
		
\vskip 2mm
		
In {\bf Case-iii}, when ${\rm dim}\br{V_{\frac{1}{\sqrt{2}}}}+{\rm dim}\br{V_{-\frac{1}{\sqrt{2}}}}=m+n-2$, then $\tilde{g}=2$ and $g\leq3$. 
		
When ${\rm dim}\br{V_{\frac{1}{\sqrt{2}}}}+{\rm dim}\br{V_{-\frac{1}{\sqrt{2}}}}\leq m+n-3$, for any other $\lambda\in \sigma(A|_{V^\perp})$, it follows from Corollary \ref{coro:4.2} (2) that $TV_\lambda \perp V_\lambda$. Now, for any unit vector field $X\in V_{\frac{1}{\sqrt{2}}}$, by \eqref{eqn:car-0-general},  
we get
\begin{align*}
0 = & \sum_{i=1, \mu_i\neq\frac{1}{\sqrt{2}}}^{m+n-2}\frac{\frac{1}{\sqrt{2}}\mu_i-\frac{1}{2}}{\frac{1}{\sqrt{2}}-\mu_i} \big(1+\langle TX,X\rangle\langle Te_i,e_i\rangle-2\langle TX,e_i\rangle^2\big)\\
=& -\trh \left(m+n-2- \operatorname{dim}\br{V_{\frac{1}{\sqrt{2}}}}+ \sum_{\mu_i =-\frac{1}{\sqrt{2}}} \lang{TX, X}\lang{Te_i ,e_i} -2\sum_{\mu_i\neq\frac{1}{\sqrt{2}}} \langle TX,e_i\rangle^2\right).
\end{align*}
For the above vector field $X\in V_{\frac{1}{\sqrt{2}}}$, let 
$|\lang{TX, X}|=a$, then
\begin{align*}
0 = & m+n-2- \operatorname{dim}\br{V_{\frac{1}{\sqrt{2}}}}+ \sum_{\mu_i =-\frac{1}{\sqrt{2}}} \lang{TX, X}\lang{Te_i ,e_i} -2\sum_{\mu_i\neq\frac{1}{\sqrt{2}}} \langle TX,e_i\rangle^2\\
\geq &m+n-2 - \operatorname{dim}\br{V_{\frac{1}{\sqrt{2}}}} - \operatorname{dim}\br{V_{-\frac{1}{\sqrt{2}}}}a- 2(1-\sum_{\mu_i=\frac{1}{\sqrt{2}}} \langle TX,e_i\rangle^2)\\
\geq & m+n-2- \operatorname{dim}\br{V_{\frac{1}{\sqrt{2}}}}- \operatorname{dim}\br{V_{-\frac{1}{\sqrt{2}}}} - 2(1-a^2) + (1 - a)\operatorname{dim}\br{V_{-\frac{1}{\sqrt{2}}}}.
\end{align*}
If ${\rm dim}\br{V_\trh} + {\rm dim}\br{V_{-\trh}} \leq m+n-4$, then
\begin{align*}
0 \geq & m+n-2 - {\rm dim}\br{V_{\frac{1}{\sqrt{2}}}}- {\rm dim}\br{V_{-\frac{1}{\sqrt{2}}}} - 2 + 2a^2 + (1-a){\rm dim}\br{V_{-\frac{1}{\sqrt{2}}}}\\
\geq & 2a^2 + (1-a){\rm dim}\br{V_{-\frac{1}{\sqrt{2}}}}\geq 2a^2 + 1 - a > 0,
\end{align*}
which contradicts $0\leq a\leq 1$. Then ${\rm dim}\br{V_{\trh}}+{\rm dim}\br{V_{-\trh}} =m+n-3$ and $0\geq -1 + 2a^2 + (1-a){\rm dim}\br{V_{-\trh}}$.
If ${\rm dim}\br{V_{-\trh}}\geq 2$, then
$$
0\geq 2a^2 - 1 + (1-a){\rm dim}\br{V_{-\trh}}\geq 2a^2 - 2a + 1 =a^2+(1-a)^2,
$$
which contradicts $0\leq a\leq 1$. Thus   ${\rm dim}\br{V_{-\trh}} = 1$. 
		
On the other hand, for any unit vector field $X\in V_{-\frac{1}{\sqrt{2}}}$, by \eqref{eqn:car-0-general},  we get 
\begin{align*}
0 = & \sum_{i=1, \mu_i\neq-\frac{1}{\sqrt{2}}}^{m+n-2}\frac{-\frac{1}{\sqrt{2}}\mu_i-\frac{1}{2}}{-\frac{1}{\sqrt{2}}-\mu_i} \big(1+\langle TX,X\rangle\langle Te_i,e_i\rangle-2\langle TX,e_i\rangle^2\big)\\
=& \trh \left(m+n-2-{\rm dim}\br{V_{-\frac{1}{\sqrt{2}}}}+ \sum_{\mu_i =\frac{1}{\sqrt{2}}} \lang{TX, X}\lang{Te_i ,e_i} -2\sum_{\mu_i\neq-\frac{1}{\sqrt{2}}} \langle TX,e_i\rangle^2\right).
\end{align*}
By a procedure similar to the one used to prove that ${\rm dim}\br{V_{-\trh}} = 1$, we can obtain that ${\rm dim}\br{V_{\trh}} = 1$. 
Since ${\rm dim}\br{V_{\trh}}+{\rm dim}\br{V_{-\trh}} =m+n-3$, it follows that $\sigma(A|_{V^\perp}) = \{\trh,-\trh, \lambda\}$, where $\lambda\neq\trh$, $\lambda\neq-\trh$ and ${\rm dim}(V_{\trh}) = {\rm dim}(V_{-\trh}) ={\rm dim}(V_{\lambda}) = 1$. 
		
In the following, we take a principal orthonormal basis of $V^\perp$ such that 
$Ae_1=\trh e_1$, $Ae_2=-\trh e_2$ and $Ae_3=\lambda e_3$.  
Substituting $X=e_1$ into \eqref{eqn:car-0-general} again, we get
\begin{equation}\label{eqn:6.ds1}
\begin{aligned}
0 = & -\trh \big(2+ \lang{Te_1, e_1}\lang{Te_2 ,e_2} -2(\langle Te_1,e_2\rangle^2+\langle Te_1,e_3\rangle^2)\big)\\
= & -\trh \big(2+ \lang{Te_1, e_1}\lang{Te_2 ,e_2} - 2 + 2\lang{Te_1, e_1}^2\big)\\
= & -\trh \big(\lang{Te_1, e_1}\lang{Te_2 ,e_2} + 2\lang{Te_1, e_1}^2\big).
\end{aligned}
\end{equation}
Similarly, substituting $X=e_2$ into \eqref{eqn:car-0-general} again, 
we get 
\begin{equation}\label{eqn:6.ds2}
\lang{Te_1, e_1}\lang{Te_2 ,e_2} + 2\lang{Te_2, e_2}^2= 0.
\end{equation}
By \eqref{eqn:6.ds1} and \eqref{eqn:6.ds2}, we conclude 
$$
\lang{Te_1 ,e_1} = \lang{Te_2 ,e_2} = 0.
$$
Since $\lambda^2\neq \frac{1}{2}$, Corollary \ref{coro:4.2} (2) tells us that $\lang{Te_3,e_3} = 0$, so ${\rm Tr}(T|_{V^\perp}) = \lang{Te_1 ,e_1} + \lang{Te_2 ,e_2} + \lang{Te_3 ,e_3} = 0$. But from ${\rm dim}(V^\perp) =m+n-2=3$, we have ${\rm Tr}(T|_{V^\perp})=m-n\neq 0$. 
It is a contradiction. 
\end{proof}
	
\begin{proposition}\label{prop:C=0}
Let $M$ be a connected oriented hypersurface of $\mathbb{H}^m \times \mathbb{H}^n$ ($m\geq3, n\geq2$) with constant principal curvatures and product angle function $C\equiv0$. Then $g\in\{2,3\}$, and up to an isometry of $\mathbb{H}^m \times \mathbb{H}^n$, one of the following three cases occurs: 
\begin{enumerate}
\item $M$ is an open part of $M_{1,-1}^0$ (see Example \ref{exam:3}); or
\item $M$ is an open part of $M_{1,1}^0$ (see Example \ref{exam:4}); or
\item $m=n$, and $M$ is an open part of $M_\tau$ for some $\tau<-1$ (see Example \ref{exam:5}). 
\end{enumerate}
\end{proposition}
	
\begin{proof}
Previously, we proved that $\tilde{g}\leq 2$. If $0\in \sigma(A|_{V^\perp})$ or $\tilde{g} = 1$,  then $g\leq 2$. By Theorem \ref{thm:1.2} and $C=0$, $M$ is an open part of $M_{1,-1}^{0}$ (see Example \ref{exam:3}). In the following, we only need to consider the situation that $g = 3$. Since $\tilde{g}\leq 2$, we can write $\sigma(A|_{V^\perp}) = \{\lambda, \mu\}$, where $\lambda, \mu\neq 0$ and $\lambda\neq\mu$. By $C = 0$ and \eqref{eqn:car-0-2}, for any unit vector fields $X\in V_\lambda$, $Y\in V_\mu$, we have
\begin{equation}\label{eqn:6.3-1}
\br{\lambda\mu - \frac{1}{2}}\br{1 + \lang{TX,X}\lang{TY,Y} - 2\lang{TX, Y}^2} = 0.
\end{equation}
We now divide the discussion into two cases: 
		
{\bf Case-i}: $\sigma(A|_{V^\perp})=\{\frac{1}{\sqrt 2},-\frac{1}{\sqrt 2}\}$; 
		
{\bf Case-ii}: At least one of $\{\frac{1}{\sqrt 2},-\frac{1}{\sqrt 2}\}$ does not belong to $\sigma(A|_{V^\perp})$. 
		
In {\bf Case-i}, let $\lambda = -\frac{1}{\sqrt{2}}$ and $\mu = \frac{1}{\sqrt{2}}$. 
Without loss of generality, we suppose that ${\rm dim}(V_\mu)\geq {\rm dim}(V_\lambda)$. If ${\rm dim}(V_\mu) = 1$, then ${\rm dim}(V_\lambda) = 1$ and ${\rm dim}(V^\perp) = m + n - 2 = 2$, which contradicts $m\geq3$ and $n\geq2$. 
		
If ${\rm dim}(V_\mu)\geq 2$, then for any unit vector field $\tilde{X}\in V_\lambda$, we can choose some unit vector field $\tilde{Y}\in V_\mu$ such that $\lang{T\tilde{X},\tilde{Y}} = 0$. It follows from \eqref{eqn:6.3-1} that 
\begin{equation}\label{eqn:6.z1}
1 + \lang{T\tilde{X},\tilde{X}}\lang{T\tilde{Y},\tilde{Y}} = 0. 
\end{equation}
From the fact that $\|T\tilde{X}\|=\|T\tilde{Y}\|=1$, by \eqref{eqn:6.z1}, we have 
$$
\lang{T\tilde{X},\tilde{X}} = \pm 1,\ \text{ $\forall \tilde{X}\in V_\lambda$ and $\|\tilde{X}\|=1$}.
$$
It follows that $TV_\lambda=V_\lambda$ and therefore $TV_\mu = V_\mu$. 
Thus, we know that $M$ satisfies $AT=TA$. 
By Proposition \ref{prop:ATTA}, $M$ is an open part of $M_{1,1}^0$. 
		
\vskip 2mm
		
In {\bf Case-ii}, without loss of generality, we assume that $\lambda\neq \pm\trh$. By Corollary \ref{coro:4.2} (2), we have $TV_\lambda \perp V_\lambda$. 
For any unit vector $\tilde {X} \in V_\lambda$, we have $\lang{T\tilde{X},\tilde{X}}=0$. Substituting $(\tilde X, T\tilde X)$ into \eqref{eqn:6.3-1}, we get 
$$
0=\br{\lambda\mu - \frac{1}{2}}(1 - 2\lang{T\tilde{X},T\tilde{X}}^2) =-\br{\lambda\mu - \frac{1}{2}}.
$$
Thus $\mu = \frac{1}{2\lambda}\neq\pm\trh$. 
By Corollary \ref{coro:4.2} (2), we also have 
$TV_\mu\perp V_\mu$. 
Since $V^\perp=V_\lambda\oplus V_\mu$ and $T|_{V^\perp}$ is an orthogonal involution, we obtain $T(V_\lambda)=V_\mu$ and $T(V_\mu)=V_\lambda$, 
hence ${\rm Tr}(T|_{V^\perp}) = 0$. Since ${\rm Tr}(T|_{V^\perp})=m-n$, we have $m=n$.
In particular, $\dim V_\lambda=\dim V_\mu=m-1$.
		
Now, without loss of generality, we suppose $0<\mu< \frac{1}{\sqrt{2}} < \lambda$. 
Let $\Phi_r(M)$ be the family of parallel hypersurfaces of $M$ defined by 
$$
\begin{aligned}
\Phi_r:\ M &\longrightarrow \mathbb{H}^m\times \mathbb{H}^m,\\
x&\longmapsto {\rm exp}_x(rN_x),  
\end{aligned}
$$
where ${\rm exp}$ denotes the exponential map in 
$\mathbb{H}^m\times \mathbb{H}^m$, $N_x$ is the unit normal vector at $x\in M$. 
		
Let $r_0=\sqrt{2}{\rm arccoth}(\sqrt{2}\lambda)$, then 
$\cosh \br{\frac{r_0}{\sqrt{2}}}-\sqrt{2}\lambda \sinh \br{\frac{r_0}{\sqrt{2}}}=0$.  
Since $C=0$ on $M$, by Proposition \ref{pro:7.1a1}, 
we have 
$$
{\rm rank}(d\Phi_{r_0})=m,\quad T\Phi_{r_0}(M)=\mathcal{P}_{r_0}(V_{\mu}\oplus \mathbb{R}V),
\quad T^\perp \Phi_{r_0}(M)=\mathcal{P}_{r_0}(V_{\lambda}\oplus \mathbb{R}N),  
$$
where $\mathcal{P}_{r_0}$ is the parallel transport along the normal geodesic to $\Phi_{r_0}(M)$. 
Then $\Phi_{r_0}(M)$ is an $m$-dimensional focal submanifold of $M$. 
Since $P(V_{\mu}\oplus \mathbb{R}V)=V_{\lambda}\oplus \mathbb{R}N$ and $\bar\nabla P=0$, it follows that $PT\Phi_{r_0}(M)=T^\perp \Phi_{r_0}(M)$. Let $h$ be the second fundamental form of $\Phi_{r_0}(M)$. Then for any $X, Y, Z\in T\Phi_{r_0}(M)$, with the use of $\bar{\nabla}P=0$, we have 
\begin{equation}\label{eqn:focal1}
\langle h(X,Y),PZ   \rangle=\langle \bar{\nabla}_X Y,PZ   \rangle
=-\langle \bar{\nabla}_X PZ,Y   \rangle=-\langle \bar{\nabla}_X Z,PY   \rangle=-\langle h(X,Z),PY   \rangle. 
\end{equation}
It follows from \eqref{eqn:focal1} that 
$$
\begin{aligned}
\langle h(X,Y),PZ   \rangle&=-\langle h(X,Z),PY   \rangle=-\langle h(Z,X),PY   \rangle=
\langle h(Z,Y),PX   \rangle \\
&=\langle h(Y,Z),PX   \rangle=-\langle h(Y,X),PZ   \rangle=
-\langle h(X,Y),PZ   \rangle. 
\end{aligned}
$$
Hence $\langle h(X,Y),PZ   \rangle=0$ for any $X, Y, Z\in T\Phi_{r_0}(M)$, 
and $\Phi_{r_0}(M)$ is totally geodesic in $\mathbb{H}^m\times \mathbb{H}^m$. 
Then by Corollary 1.3 in \cite{Toj13}, up to an ambient isometry, $\Phi_{r_0}(M)$ is an open part of $\{(p,p)\in \mathbb{H}^m\times \mathbb{H}^m\}$.  Moreover, $M$ is a tube around $\Phi_{r_0}(M)$. Thus $M$ is an open part of $M_\tau$ for some $\tau<-1$ (see Example \ref{exam:5}).
\end{proof}
	
Finally, we present the following classification result. 
	
\begin{theorem}\label{thm:1.1}
Let $M$ be a connected oriented hypersurface of $\mathbb{H}^m\times \mathbb{H}^n$ ($m\geq3, n\geq2$) with constant principal curvatures and constant product angle function $C$. Then, $g\in\{1,2,3\}$, and up to isometries of $\mathbb{H}^m\times \mathbb{H}^n$, one of the following five cases occurs: 
\begin{itemize}
\item[(1)] $M$ is an open part of $\Sigma\times \mathbb{H}^n$, where $\Sigma$ is a hypersurface of $\mathbb{H}^m$ with constant principal curvatures (see Example \ref{exam:1}); or
\item[(2)] $M$ is an open part of $\mathbb{H}^m\times \tilde{\Sigma}$, where $\tilde{\Sigma}$ is a hypersurface of $\mathbb{H}^n$ with constant principal curvatures (see Example \ref{exam:2}); or
\item[(3)] $M$ is an open part of $M_{1,-1}^{C}$ for some $C\in(-1,1)$ (see Example \ref{exam:3}); or
\item[(4)] $M$ is an open part of $M_{1,1}^{C}$ for some $C\in (-1,1)$ (see Example \ref{exam:4}); or
\item[(5)] $m=n$, and $M$ is an open part of $M_\tau$ for some $\tau<-1$ (see Example \ref{exam:5}). 
\end{itemize}
\end{theorem}
	
\begin{proof}
If $C^2=1$, then Lemmas \ref{lemma:3.5}--\ref{lemma:3.2aa} give cases (1) and (2). If $0<|C|<1$, then Proposition \ref{prop:Cneq0} gives cases (3) and (4). If $C=0$, then Proposition \ref{prop:C=0} gives cases (3)--(5). This completes the proof. 
\end{proof}

\begin{remark}
Since $\mathbb{H}^m\times \mathbb{H}^1$ is locally isometric to $\mathbb{H}^m\times \mathbb{R}$, the classifications in $\mathbb{H}^m\times \mathbb{H}^1$ of hypersurfaces with at most two distinct constant principal curvatures and of those with constant principal curvatures and constant product angle function follow from \cite{CS19}. The corresponding results in $\mathbb{H}^2\times\mathbb{H}^2$ were obtained 
by \cite{GMY24}.
\end{remark}
	
\section{Proof of Theorem \ref{thm:isoparametric-HH}}\label{sect:7}

In this section, we first apply the two-parameter Jacobi field formula (see Proposition \ref{prop:two-parameter-Jacobi}) to prove that every isoparametric hypersurface in $\mathbb{H}^m\times \mathbb{H}^n$ ($m\geq3, n\geq2$) has constant principal curvatures. We then use this result to complete the proof of Theorem \ref{thm:isoparametric-HH}.
	
\begin{proposition}\label{prop:constant-angle-isoparametric}
Let $M$ be a connected oriented isoparametric hypersurface of $\mathbb H^m\times\mathbb H^n$ ($m\geq3, n\geq2$). Then $M$ has constant principal curvatures. 
\end{proposition}
	
\begin{proof}
By Theorem 1 of \cite{DP}, the product angle function $C$ of $M$ is constant.
In the following, we divide the argument into three cases: 
			
{\bf Case-1}: $C^2\equiv1$ on $M$; 
		
{\bf Case-2}: $C\equiv0$ on $M$; 
		
{\bf Case-3}: $0<|C|<1$ on $M$.  
		
In {\bf Case-1}, Lemma \ref{lemma:3.5} locally splits $M$ as $\Sigma\times\mathbb H^n$ or $\mathbb H^m\times\widetilde\Sigma$. Lemmas \ref{lemma:3.2} and \ref{lemma:3.2aa} then show directly that $M$ has constant principal curvatures. 
		
\vskip2mm 
		
In {\bf Case-2}, let $\Phi_t(M)$ be a nearby parallel hypersurface of $M$ defined by 
\eqref{eqn:parallel}, and set $\tilde t=\frac1{\sqrt2}\tanh\frac t{\sqrt2}$. 
At point $x\in M$, we denote the eigenvalues of $A$ restricted to $V^\perp$, counted with multiplicity, as follows:  
$$
\lambda_1(x), \dots, \lambda_{m+n-2}(x).
$$
Then by \eqref{eqn:7.1a3}, for any $x\in M$, the principal curvatures of parallel hypersurface 
$\Phi_t(M)$ at $\Phi_t(x)$ are given by 
$$
0, \ \ \frac{\lambda_i(x)-\tilde t}{1-2\lambda_i(x) \tilde t}
=-\tilde t+(1-2\tilde t^2)\frac{\lambda_i(x)}{1-2\lambda_i(x) \tilde t}, \ \ 1\leq i\leq m+n-2.
$$
Therefore the mean curvature $H(t)={\rm tr}A_t$ of $\Phi_t(M)$ is
\begin{equation}\label{eqn:7.ds1}
H(t)=-(m+n-2)\tilde t+(1-2\tilde t^2)\sum_{i=1}^{m+n-2}
\frac{\lambda_i(x)}{1-2\lambda_i(x)\tilde t}.
\end{equation}
		
Since $M$ is isoparametric, by definition the mean curvature of every
nearby parallel hypersurface $\Phi_t(M)$ is constant on $\Phi_t(M)$. Hence the right-hand side of \eqref{eqn:7.ds1} is independent of $x$ for each fixed small $t$.
It follows that the function $F(\tilde t,x):=\sum_{i=1}^{m+n-2}\frac{\lambda_i(x)}{1-2\lambda_i(x)\tilde t}$ 
is independent of $x$ for every sufficiently small $\tilde t$. 
Consequently, for every integer $k\ge0$, the derivative
$\partial_{\tilde t}^kF(0,x)$ is independent of $x$. Expanding in powers of $\tilde t$, 
$$
F(\tilde t,x)=\sum_{i=1}^{m+n-2}\lambda_i(x)
\sum_{k=0}^\infty\bigl(2\lambda_i(x)\tilde t\bigr)^k
=\sum_{k=0}^\infty2^k p_{k+1}(x)\,{\tilde t}^k,
$$
where
$$
p_{k+1}(x)=\sum_{i=1}^{m+n-2}\lambda_i(x)^{k+1}
={\rm tr}\left(A(x)^{k+1}\right),\qquad k+1\ge1.
$$
Here the zero eigenvalue in the direction of $\bar{V}$ does not contribute.
Thus $\partial_{\tilde t}^kF(0,x)=2^k k!\,p_{k+1}(x)$, 
and therefore $p_{k+1}(x)$ is constant on $M$ for every $k\ge0$.
By Newton's identities,
$$
l\sigma_l(x)=\sum_{i=1}^{l}(-1)^{i-1}\sigma_{l-i}(x)\,p_i(x),
\qquad \sigma_0(x)=1,
$$
where
$$
\sigma_l(x)=\sum_{1\le i_1<\cdots<i_l\le m+n-2}
\lambda_{i_1}(x)\cdots\lambda_{i_l}(x),
\qquad 1\le l\le m+n-2.
$$
Hence every elementary symmetric function $\sigma_l(x)$ is a polynomial
in $p_1(x),\dots,p_l(x)$, and is therefore constant on $M$. It follows
that the characteristic polynomial of $A(x)|_{V^\perp_x}$ is independent of $x$, so the multiset
$$
\{\lambda_1(x),\dots,\lambda_{m+n-2}(x)\}
$$
is independent of $x$. The remaining principal curvature in the direction of $V$ is $0$. Since $A$ is a smooth symmetric endomorphism, its ordered
eigenvalues are continuous functions on $M$. Each such continuous
function takes values in the fixed finite multiset above. As $M$ is
connected, every such function must be constant. Therefore all
principal curvatures of $M$ are constant. 
		
\vskip 2mm 
		
In {\bf Case-3}, we have $\alpha\neq\beta$. Fix $o\in M$ and retain the notation of Proposition \ref{prop:two-parameter-Jacobi}. We define 
\begin{equation}\label{eqn:7.1sssa1}
Q_o(x,y)=\det\left[
\begin{pmatrix}
(x+1)I_{m-1}&0\\
0&(y+1)I_{n-1} 
\end{pmatrix}
-\begin{pmatrix}
\dfrac{x-1}{\alpha}I_{m-1}&0\\
0&\dfrac{y-1}{\beta}I_{n-1} 
\end{pmatrix}S_o
\right].
\end{equation}
By $\cosh r=\frac12 e^{-r}(e^{2r}+1)$, $\sinh r=\frac12e^{-r}(e^{2r}-1)$ and \eqref{eqn:7.1a1}, we have 
\begin{equation}\label{eqn:7.1a4as1}
d_o(t,s)=2^{-(m+n-2)}e^{-(m-1)r-(n-1)w}Q_o(e^{2r},e^{2w}),
\qquad
Q_o(1,1)=2^{(m+n-2)}.
\end{equation}
We first prove that the separation identity \eqref{eqn:7.1dd1} forces $Q_o$ to be a monomial. The following two elementary observations will be used. 
		
\textbf{Claim (1).}  For any real symmetric matrix $S_o$ and any $\alpha, \beta>0$,
$Q_o(x,y)\neq 0$ whenever $\IIm x>0$ and $\IIm y>0$.
		
\noindent\textit{Proof of the claim (1).}  		
If $\IIm x, \IIm y>0$ then $x, y\neq\pm1$, and
\begin{equation}\label{eqn:7.1a4}
Q_o(x,y)=(x+1)^{m-1}(y+1)^{n-1}\;\det\Bigl(
I_{m+n-2}-\diag\Bigl(\tfrac{x-1}{\alpha(x+1)}I_{m-1},\,
\tfrac{y-1}{\beta(y+1)}I_{n-1}\Bigr)S_o\Bigr). 
\end{equation}
If $Q_o(x,y)=0$ for some $\IIm x>0$ and $\IIm y>0$, there is $u\neq0$ with
\begin{equation}
S_ou=Zu,\quad\quad
Z=\diag\Bigl(\alpha\tfrac{x+1}{x-1}I_{m-1},\;\beta\tfrac{y+1}{y-1}I_{n-1}\Bigr). 
\end{equation}
By the Cayley identity,
$\IIm\frac{x+1}{x-1}=-\frac{2\,\IIm x}{|x-1|^2}<0$, and likewise for $y$. Hence
$\IIm Z$ is negative definite. Since $S_o=S_o^{*}$, we have $u^{*}S_ou\in\RR$, and from
$(Z-S_o)u=0$ we get the contradiction
$$
0=\IIm\bigl(u^{*}(Z-S_o)u\bigr)=\IIm(u^{*}Zu)=u^{*}(\IIm Z)\,u<0 . 
$$
Therefore $Q_o(x,y)\neq0$ whenever $\IIm x,\IIm y>0$.  
\qed

\vskip1mm 

\textbf{Claim (2).} Let $a,b$ be positive integers with $a\neq b$ (so $a+b\ge 3$), and let
$k\in\CC\setminus\{0\}$. Then:
\begin{enumerate}
\item[(i)] the equation $x^{a}y^{b}=k$ has a solution with $\IIm x>0$,
$\IIm y>0$;
\item[(ii)] the equation $x^{b}=k\,y^{a}$ has a solution with $\IIm x>0$,
$\IIm y>0$.
\end{enumerate}

\noindent\textit{Proof of the claim (2).}  		
Write $x=R_1e^{i\theta}$, $y=R_2e^{i\varphi}$ with $R_1,R_2>0$. 

For (i) we need $a\theta+b\varphi\equiv\arg k\pmod{2\pi}$ with
$0<\theta,\varphi<\pi$. As $(\theta,\varphi)$ ranges over $(0,\pi)^2$, the quantity
$a\theta+b\varphi$ ranges over the interval $(0,(a+b)\pi)$, whose length is 
$(a+b)\pi>2\pi$, so the congruence is solvable. We then choose $R_1,R_2>0$ with $R_1^{a}R_2^{b}=|k|$. 

For (ii) we need $b\theta-a\varphi\equiv\arg k\pmod{2\pi}$. The range of
$b\theta-a\varphi$ is $(-a\pi,b\pi)$, again of length $(a+b)\pi>2\pi$, and then
we choose $R_1,R_2>0$ with $R_1^{b}/R_2^{a}=|k|$. 
\qed

\vskip 2mm 

In what follows, we shall prove the monomial identity \eqref{eqn:7.1a10} by dividing the argument into two cases: 

{\bf Case-3-i}: $\alpha/\beta=a/b$ with coprime positive integers
$a\neq b$; 

{\bf Case-3-ii}: $\alpha/\beta\notin\mathbb{Q}:=\left\{ \frac{a}{b}|\, a, b \in \mathbb{Z}, \ b \neq 0 \right\}$. 

In {\bf Case-3-i}, set $\gamma=1/\sqrt{a^2+b^2}$, so that
$\alpha=a\gamma$ and $\beta=b\gamma$. 
By \eqref{eqn:7.1a4as1} with $s=0$, we have 
\begin{equation}\label{eqn:7.1a5}
\begin{aligned}
f(t)&=d_o(t,0)=2^{-(m+n-2)}e^{-((m-1)\alpha+(n-1)\beta)t}
Q_o\bigl(e^{2a\gamma t},e^{2b\gamma t}\bigr)\\
&=2^{-(m+n-2)}e^{-((m-1)\alpha+(n-1)\beta)t}P(\hat{z}),
\qquad \hat{z}=e^{2\gamma t}, 
\end{aligned}
\end{equation}
where $P(\hat{z})=Q_o(\hat{z}^{a},\hat{z}^{b})$. 	
Let $\hat{l}=\operatorname{ord}_0 P\ge 0$ be the order of the zero of $P$ at $\hat{z}=0$.
Write $P=\hat{z}^{\hat{l}}\widetilde P$ with
$\widetilde P(0)\neq0$. Suppose $\deg\widetilde P\ge1$ and let $z_0\neq0$ be a root
of $\widetilde P$. Choose $t_0$ with $e^{2\gamma t_0}=z_0$. Then $f(t_0)=0$, and
by \eqref{eqn:7.1dd1}, $d_o(t_0,s)=0$ for all $s\in \mathbb{C}$. Put
$x_0=e^{2\alpha t_0}$, $y_0=e^{2\beta t_0}$, and $z=e^{2\gamma s}$. By \eqref{eqn:7.1a4as1}, we get 
\begin{equation}\label{eqn:7.1a6}
Q_o\bigl(x_0z^{b},\,y_0z^{-a}\bigr)=0\qquad\text{for all }z\in\CC^{*}. 
\end{equation}
Note that the level set $\{(x,y)\in \mathbb{C}^2|\,x^{a}y^{b}=k\}$, where
$k=x_0^{a}y_0^{b}\neq0$, is exactly the image of the parametrization
$z\mapsto(x_0z^{b},y_0z^{-a})$. Indeed, let $x^{a}y^{b}=k$ and set
$u=x/x_0$, $v=y/y_0$, so that $u^{a}v^{b}=1$. Choose $z_1$ with $z_1^{b}=u$.
Then $(vz_1^{a})^{b}=v^{b}u^{a}=1$, so $vz_1^{a}=\zeta\in\mu_b$, the group of
$b$-th roots of unity. Since $\gcd(a,b)=1$, the map
$\mu_b\to\mu_b$, $\zeta'\mapsto\zeta'^{a}$ is a bijection; choose $\zeta'$ with
$\zeta'^{a}=\zeta^{-1}$ and put $z=z_1\zeta'$. Then $z^{b}=z_1^{b}=u$ and
$z^{-a}=z_1^{-a}\zeta'^{-a}=z_1^{-a}\zeta=v$. Hence $Q_o$ vanishes on the whole
level set $\{x^{a}y^{b}=k\}$. By Claim (2) (i) this level set meets
$\{\IIm x>0,\ \IIm y>0\}$, contradicting Claim (1). 
Therefore $\deg\widetilde P=0$, so $P(\hat{z})=c\hat{z}^{\hat{l}}$. By $f(0)=1$, we have  $f(t)=2^{-(m+n-2)}c\,e^{(2\hat{l}\gamma-(m-1)\alpha-(n-1)\beta)t}=:e^{\tilde\lambda t}$ for some constant $\tilde\lambda$. 

On the other hand, by \eqref{eqn:7.1a4as1} with $t=0$, 
\begin{equation}\label{eqn:7.1a7}
\begin{aligned}
g_o(s)&=d_o(0,s)=2^{-(m+n-2)}e^{-((m-1)\beta-(n-1)\alpha)s}
Q_o\bigl(e^{2b\gamma s},e^{-2a\gamma s}\bigr)\\
&=2^{-(m+n-2)}e^{-((m-1)\beta-(n-1)\alpha)s}R(\tilde{z}),
\qquad \tilde{z}=e^{2\gamma s}, 
\end{aligned}
\end{equation}
where $R(\tilde{z})=Q_o(\tilde{z}^{b},\tilde{z}^{-a})$ is a Laurent polynomial. If 
$z_0\neq0$ is a root of $R$, we choose $s_0$ with $e^{2\gamma s_0}=z_0$. Then 
$g_o(s_0)=0$, and by \eqref{eqn:7.1dd1}, $d_o(t,s_0)=0$ for all $t\in \mathbb{C}$. Put
$x_0=e^{2\beta s_0}$, $y_0=e^{-2\alpha s_0}$, $z=e^{2\gamma t}$. By \eqref{eqn:7.1a4as1}, we get 
\begin{equation}\label{eqn:7.1a8}
Q_o\bigl(x_0z^{a},\,y_0z^{b}\bigr)=0\qquad\text{for all }z\in\CC^{*}. 
\end{equation}
Note that the level set $\{(x,y)\in (\mathbb{C}^*)^2|\,x^{b}=ky^{a}\}$,
$k=x_0^{b}/y_0^{a}\neq0$, is exactly the image of the parametrization 
$z\mapsto(x_0z^{a},y_0z^{b})$. Indeed, for $x^{b}=ky^{a}$, set $u=x/x_0$, $v=y/y_0$, so $u^{b}=v^{a}$. Choose
$z_1\in\CC^{*}$ with $z_1^{a}=u$. Then, $\bigl(v\,z_1^{-b}\bigr)^{a}=v^{a}z_1^{-ab}=v^{a}u^{-b}=1$, 
so $v\,z_1^{-b}=\zeta\in\mu_a$, the group of
$a$-th roots of unity. 
Since $\gcd(a,b)=1$, the power map $\mu_a\to\mu_a$, $\zeta'\mapsto\zeta'^{b}$
is a bijection; choose $\zeta'\in\mu_a$ with $\zeta'^{b}=\zeta$ and set
$z=z_1\zeta'$. Then $z^{a}=z_1^{a}(\zeta')^{a}=u$, and 
$z^{b}=z_1^{b}(\zeta')^{b}=z_1^{b}\,\zeta
=z_1^{b}\cdot\bigl(v\,z_1^{-b}\bigr)=v$. Hence $Q_o$ vanishes on $\{x^{b}=ky^{a}\}$, which by
Claim (2) (ii) meets $\{\IIm x>0,\ \IIm y>0\}$, contradicting
Claim (1). Thus $R(\tilde{z})=c\tilde{z}^{\tilde{l}}$ is a monomial for some $\tilde{l}\in \mathbb{Z}$. Then by $g_o(0)=1$, we write $g_o(s)=e^{\tilde\mu s}$ for some constant $\tilde\mu$. 

Combining \eqref{eqn:7.1dd1} and \eqref{eqn:7.1a4as1} with $f(t)=e^{\tilde\lambda t}$, $g_o(s)=e^{\tilde\mu s}$, and
inverting $(t,s)\mapsto(r,w)$ by $t=\alpha r+\beta w$,
$s=\beta r-\alpha w$, we obtain
\begin{equation}\label{eqn:7.1a9}
Q_o\bigl(e^{2r},e^{2w}\bigr)
=2^{m+n-2}e^{(m-1+\alpha\tilde\lambda+\beta\tilde\mu)r+(n-1+\beta\tilde\lambda-\alpha\tilde\mu)w}. 
\end{equation}
Write $Q_o(x,y)=\sum_{i=0}^{m-1}\sum_{j=0}^{n-1}c_{ij}x^{i}y^{j}$, then 
$Q_o(e^{2r},e^{2w})=\sum c_{ij}e^{2ir+2jw}$, which is a finite linear combination of the distinct functions $e^{2ir+2jw}$. 
Their linear independence and \eqref{eqn:7.1a9}, $Q_o(1,1)=2^{m+n-2}$ show that 
\begin{equation}\label{eqn:7.1a10}
Q_o(x,y)=2^{m+n-2}x^{i}y^{j}\qquad\text{for some }0\le i\le m-1,\ 0\le j\le n-1 .
\end{equation}

\vskip2mm 

In {\bf Case-3-ii}, we write $Q_o(x,y)=\sum_{i=0}^{m-1}\sum_{j=0}^{n-1}c_{ij}x^{i}y^{j}$. From \eqref{eqn:7.1a4as1}, we have 
$$
d_o(t,s)=2^{-(m+n-2)}\sum_{i,j}c_{ij}e^{\lambda_{ij}t+\mu_{ij}s},
$$
where
\begin{equation}\label{eqn:7.1a11}
\lambda_{ij}=(2i-(m-1))\alpha+(2j-(n-1))\beta,\qquad
\mu_{ij}=(2i-(m-1))\beta-(2j-(n-1))\alpha . 
\end{equation}
By \eqref{eqn:7.1dd1} and the evaluation at $s=0$, we get 
\begin{equation}\label{eqn:7.1a12}
\sum_{i,j}c_{ij}e^{\lambda_{ij}t}\bigl(e^{\mu_{ij}s}-g_o(s)\bigr)=0
\qquad\text{for all }(t,s). 
\end{equation}
Since $\alpha/\beta\notin\mathbb{Q}$, the numbers $\lambda_{ij}$ are pairwise
distinct. Fixing $s$ and using the linear independence of the exponentials
$e^{\lambda_{ij}t}$, we get
$c_{ij}\bigl(e^{\mu_{ij}s}-g_o(s)\bigr)=0$ for every $(i,j)$ and every $s$.
If $c_{ij}\neq0$ then $e^{\mu_{ij}s}=g_o(s)$ for all $s$. If two coefficients were
nonzero, say $c_{ij},c_{i'j'}\neq0$, then $\mu_{ij}=\mu_{i'j'}$, i.e.
$\beta(2i-2i')=\alpha(2j-2j')$, which forces $(i,j)=(i',j')$ by
irrationality. Hence exactly one coefficient is nonzero, and
$Q_o(1,1)=c_{ij}=2^{m+n-2}$ still gives \eqref{eqn:7.1a10}. 

\vskip 2mm 

Now, we translate the monomial identity \eqref{eqn:7.1a10} back to the shape operator. Set
$$
U=\alpha\frac{x+1}{x-1},\qquad W=\beta\frac{y+1}{y-1}, 
$$
so that
$$
x=\frac{U+\alpha}{U-\alpha},\qquad x-1=\frac{2\alpha}{U-\alpha},\qquad
\frac{x-1}{\alpha}\,U=x+1,
$$
and similarly for $y$. By \eqref{eqn:7.1sssa1} and \eqref{eqn:7.1a10}, we have 
\begin{equation}\label{eqn:7.1a13}
\begin{aligned}
&Q_o(x,y)
=\Bigl(\frac{x-1}{\alpha}\Bigr)^{m-1}\Bigl(\frac{y-1}{\beta}\Bigr)^{n-1}
\det\bigl(\diag(UI_{m-1},WI_{n-1})-S_o\bigr)
=2^{m+n-2}x^{i}y^{j},\\
&\det\bigl(\diag(UI_{m-1},WI_{n-1})-S_o\bigr)
=(U+\alpha)^{i}(U-\alpha)^{m-1-i}(W+\beta)^{j}(W-\beta)^{n-1-j}. 
\end{aligned}
\end{equation}
Taking $U=W=z$ in \eqref{eqn:7.1a13}, it follows that 
$$
\det(zI_{m+n-2}-S_o)=(z+\alpha)^{i}(z-\alpha)^{m-1-i}(z+\beta)^{j}(z-\beta)^{n-1-j}.
$$
As $S_o$ is symmetric, its eigenvalues are $-\alpha$ ($i$ times), $\alpha$ ($m-1-i$
times), $-\beta$ ($j$ times), $\beta$ ($n-1-j$ times) at point $o\in M$. 
Together with $A\bar{V}_o=0$, we have proved that every principal curvature at $o\in M$ belongs to the fixed finite set 
$$
\{0, \alpha, -\alpha, \beta, -\beta\}. 
$$
Since $o\in M$ was arbitrary, every principal curvature of $M$ takes values in the fixed finite set $\{0, \alpha, -\alpha, \beta, -\beta\}$. The ordered eigenvalues of a smooth symmetric endomorphism
depend continuously on the base point. Because $M$ is connected, each ordered principal curvature function, being continuous and taking values in a finite set, is constant. Hence all principal curvatures of $M$ are constant. 
\end{proof}

\vskip2mm

\noindent{\bf Completion of the proof of Theorem \ref{thm:isoparametric-HH}} 

Theorem 1 of \cite{DP} shows that the product angle function of any connected oriented isoparametric hypersurface $M$ in $\mathbb{H}^m \times \mathbb{H}^n$ is constant. Proposition \ref{prop:constant-angle-isoparametric} then implies that $M$ has constant principal curvatures. The desired classification is thus a consequence of Theorem \ref{thm:1.1}. Conversely, all these examples are homogeneous hypersurfaces (see Remark \ref{rem:3.1aa}) and are therefore isoparametric as well.

\vskip4mm

\noindent\textbf{Data availability.}
Data sharing is not applicable to this article as no datasets were generated or analysed during the current study.

\vskip 0.1cm
	
\noindent\textbf{Declarations}

\vskip 0.1cm
	
\noindent\textbf{Conflict of interest.} On behalf of all authors, the corresponding author states that there is no conflict of interest.


\vskip 10mm
	
\begin{flushleft}
Haizhong Li\\
{\sc Department of Mathematical Sciences, Tsinghua University,\\
Beijing, 100084, P.R. China}\\
E-mail: lihz@tsinghua.edu.cn
		
\vskip 1mm
		
Renhao Tan\\
{\sc Department of Mathematical Sciences, Tsinghua University,\\
Beijing, 100084, P.R. China}\\
E-mail: trh23@mails.tsinghua.edu.cn
		
\vskip 1mm
		
Zeke Yao\\
{\sc School of Mathematical Sciences, South China Normal University,\\
Guangzhou 510631, P.R. China}\\
E-mail: yaozk.2021@tsinghua.org.cn
		
\end{flushleft}
\end{document}